\documentclass[11pt, journal, onecolumn]{IEEEtran}

\usepackage{amsmath,amssymb,graphicx,paralist,subfigure,pdfpages,cite,algorithm,algpseudocode,paralist,url}
\usepackage{amsthm}
\usepackage{verbatim}
\usepackage{algorithm}
\usepackage{algpseudocode}
\newtheorem{lem}{Lemma} 
\newtheorem{theorem}{Theorem}

\newtheorem{assump}{Assumption}

\def\mc{\mathcal}
\def\mb{\mathbf}
\def\mbb{\mathbb}

\IEEEoverridecommandlockouts

\renewcommand{\baselinestretch}{1.5}

\def\mbb{\mathbb}
\def\mb{\mathbf}
\def\mc{\mathcal}

\def\wt{\widetilde}
\def\ol{\overline}

\def\bds{\boldsymbol}
\newcommand{\mn}[1]{{\left\vert\kern-0.25ex\left\vert\kern-0.25ex\left\vert\kern0.3ex #1 
		\kern0.3ex\right\vert\kern-0.25ex\right\vert\kern-0.25ex\right\vert}}

\begin{document}
	\title{Variance-Reduced Decentralized Stochastic Optimization with Gradient Tracking -- \\Part II:~\textbf{\texttt{GT-SVRG}}}
	\author{
		Ran Xin, Usman A. Khan, and Soummya Kar
		\thanks{
			RX and SK are with the Department of Electrical and Computer Engineering, Carnegie Mellon University, Pittsburgh, PA; {\texttt{ranx@andrew.cmu.edu, soummyak@andrew.cmu.edu}}. UAK is with the Department of Electrical and Computer Engineering, Tufts University, Medford, MA; {\texttt{khan@ece.tufts.edu}}.
			The work of RX and SK has been partially supported by NSF under grant CCF-1513936.
			The work of UAK has been partially supported by NSF under grants CCF-1350264 and CMMI-1903972. 
		}
	}
	\maketitle
	
	\begin{abstract}
	Decentralized stochastic optimization has recently benefited from gradient tracking methods~\cite{DSGT_Pu,DSGT_Xin} providing efficient solutions for large-scale empirical risk minimization problems. In Part I~\cite{GT_SAGA} of this work, we develop \textbf{\texttt{GT-SAGA}} that is based on a decentralized implementation of SAGA~\cite{SAGA} using gradient tracking and discuss regimes of practical interest where~\textbf{\texttt{GT-SAGA}} outperforms existing decentralized approaches in terms of the total number of local gradient computations. In this paper, we describe~\textbf{\texttt{GT-SVRG}} that develops a decentralized gradient tracking based implementation of SVRG~\cite{SVRG}, another well-known variance-reduction technique. We show that the convergence rate of~\textbf{\texttt{GT-SVRG}} matches that of~\textbf{\texttt{GT-SAGA}} for smooth and strongly-convex functions and highlight different trade-offs between the two algorithms in various settings.   
	\end{abstract}
	
	\noindent\fbox{%
    \parbox{\textwidth}{%
  	\begin{center}
	    {\color{red}\textbf{This is a preliminary version of the paper~\url{https://arxiv.org/abs/1912.04230}}}
	\end{center}
    }%
}

	\section{Introduction}\label{pf}
	We consider a network of~$n$ nodes 
	that cooperatively solves an optimization problem of the following form:
	\begin{align*}
	\mbox{P1}:
	\quad\min_{\mb{x}\in\mathbb{R}^p}f(\mb{x})\triangleq\frac{1}{n}\sum_{i=1}^{n}f_i(\mb{x}),
	\qquad
	f_i(\mb{x}) \triangleq \frac{1}{m_i}\sum_{j=1}^{m_i} f_{i,j}(\mb{x}).
	\end{align*}
	Each node~$i$ only processes its own local objective functions~$\{f_{i,j}:\mathbb{R}^p\rightarrow\mathbb{R}\}_{j=1}^{m_i}$. Well-known solutions for such problems include, for example, Decentralized Gradient Descent (DGD)~\cite{DGD_tsitsiklis,DGD_nedich,DGD_Yuan,GP_neidch}, dual averaging~\cite{dual_averaging_Duchi,dual_averaging_Rabbat}, ADMM~\cite{ADMM_Wei,ADMM_Shi}, EXTRA~\cite{EXTRA}, Exact Diffusion~\cite{Exact_Diffusion}, DLM~\cite{DLM}, methods based on gradient-tracking~\cite{GT_first,GT_CDC,NEXT,harnessing,DIGing,add-opt, AGT,MP_scutari,proxDGT,GT_jakovetic,Network-DANE},~$\mc{AB}$/Push-Pull~\cite{AB, push-pull,TV-AB} and dual methods~\cite{dual_optimal_ICML,dual_GT,dual_optimal_uribe}. More recently, significant effort has been made to develop stochastic variants of the aforementioned methods, for example,~\cite{DSGD_nedich,DGD_Kar,SGP_nedich,DSGD_NIPS,SGP_ICML,D2,DSGT_Pu,DSGT_Xin,SED,SGT_nonconvex_you,DSGD_vlaski_1,DSGD_vlaski_2}. These stochastic gradient methods are more favorable when each node has a large number of (local) data samples or in scenarios where each node  receives online, streaming data in real-time.	
	To better leverage the finite-sum structure of the local objective function in Problem P1, several decentralized methods have been proposed~\cite{DSA,DSBA,DAVRG,edge_DSA,ADFS,GT_SAGA} that are based on various variance-reduction techniques~\cite{SAGA,point-SAGA,AVRG,APCG}. These approaches aim to combine the advantages of both deterministic and stochastic gradient methods.

   In this two-part paper, we develop and analyze decentralized, stochastic first-order methods with the help of variance reduction techniques and gradient tracking. In Part I~\cite{GT_SAGA}, we describe~\textbf{\texttt{GT-SAGA}} that is based on SAGA~\cite{SAGA}, while in this Part II, we propose~\textbf{\texttt{GT-SVRG}}, formally described in Algorithm~\ref{gtsvrg_alg}, that is based on stochastic gradient tracking~\cite{DSGT_Pu,DSGT_Xin} and another well-known variance reduction technique called SVRG~\cite{SVRG,S2GD,SVRG-BB}.
   As in the centralized SVRG method~\cite{SVRG}, \textbf{\texttt{GT-SVRG}} has an outer-loop, indexed by~$t$, where each node~$i$ computes its local full gradient~$\wt{\nabla}_i^t$, and an inner-loop, indexed by~$k$, where decentralized stochastic gradient-tracking (type) update is recursively performed. 
   
   \begin{algorithm}[H]\label{GT-SVRG}
	\caption{~\textbf{\texttt{GT-SVRG}} at each node~$i$}
	\label{gtsvrg_alg}
	\begin{algorithmic}[1]
		\Require Arbitrary starting point~$\mb{x}_i^{0,0}\in\mbb{R}^p$ and~step-size~$\alpha>0$.
				
		~~Doubly stochastic weights:~$W=\{w_{ir}\}\in\mathbb{R}^{n\times n}$.
		
		~~Number of iterations for each inner loop:~$K>0$.
		
		~~Gradient Tracker:~$\mb{y}_i^{0,0} = \mb{v}_i^{0,0} = \widetilde{\nabla}_i^0 \triangleq \nabla f_i(\mb{x}_i^{0,0})$.
		
		\For{$t= 0,1,2,\cdots$}
		\For{$k= 0,1,2,\cdots, K-1$}
		\State $\mb{x}_{i}^{t,k+1} = \sum_{r=1}^{n}w_{ir}\mb{x}_{r}^{t,k} - \alpha\mb{y}_{i}^{t,k}$ \Comment{Inner-estimate update}
		\State Select~$s_{i}^{t,k+1}$ uniformly at random from~$\{1,\cdots,m_i\}$.
		\Comment{Sample from  local data}
		\State $\mb{v}_{i}^{t,k+1} = \nabla f_{i,s_i^{t,k+1}}(\mb{x}_{i}^{t,k+1}) - \nabla f_{i,s_i^{t,k+1}}(\mb{x}_i^{t,0}) + \widetilde{\nabla}_i^t$ \Comment{Local SVRG update}
		\State $\mb{y}_{i}^{t,k+1} = \sum_{r=1}^{n}w_{ir}\mb{y}_{r}^{t,k} + \mb{v}_i^{t,k+1} - \mb{v}_i^{t,k}$ \Comment{Gradient Tracker update}
		\EndFor
		\State~$\mb{x}_i^{t+1,0} = \mb{x}_i^{t,K}$,~$\mb{y}^{t+1,0} = \mb{y}^{t,K}$,~$\mb{v}^{t+1,0} = \mb{v}^{t,K}$.
		\Comment{Outer-update} 
		\State $\widetilde{\nabla}_i^{t+1} = \nabla f_i(\mb{x}_i^{t+1,0})$. \Comment{Compute local full gradient, fixed in the next inner loop}
		\EndFor
	\end{algorithmic}
\end{algorithm}	

We study the convergence rate of~\textbf{\texttt{GT-SVRG}} under the following assumptions. 
	\begin{assump}\label{sc}
		Each local objective,~$f_{i,j}$, is~$\mu$-strongly-convex:~$\forall\mb{x}, \mb{y}\in\mbb{R}^p$, we have, for some~$\mu>0$, 
		\begin{equation*}
		f_{i,j}(\mb{y})\geq f_{i,j}(\mb{x})+ \big\langle\nabla f_{i,j}(\mb{x}), \mb{y}-\mb{x}\big\rangle+\frac{\mu}{2}\|\mb{x}-\mb{y}\|^2.
		\end{equation*}
	\end{assump}
	We note that under Assumption 1, the global objective function~$f$ has a unique minimizer, denoted as~$\mb{x}^*$.
	\begin{assump}\label{smooth}
		Each local objective,~$f_{i,j}$, is~$L$-smooth:~$\forall\mb{x}, \mb{y}\in\mbb{R}^p$, we have, for some~$L>0$,
		\begin{equation*}
		\qquad\|\mb{\nabla} f_{i,j}(\mb{x})-\mb{\nabla} f_{i,j}(\mb{y})\|\leq L\|\mb{x}-\mb{y}\|.
		\end{equation*}
	\end{assump}
	\begin{assump}\label{connect}
		The weight matrix~$W$ associated with the network is primitive and doubly-stochastic.
	\end{assump}
\noindent The performance of \texttt{\textbf{GT-SVRG}} is described in the following theorem.

\begin{theorem}\label{main_theorem}
Let Assumption~\ref{sc}-\ref{connect} hold. Define~$\sigma$ as the second largest singular value of the network weight matrix~$W$,~$M\triangleq\max_{i}m_i$,~$m\triangleq\min_{i}m_i$ and~$Q\triangleq L/\mu$, the condition number of~$f$. Then~\textbf{\texttt{GT-SVRG}} achieves~$\epsilon$-accuracy (in terms of distance to the minimizer)~with $$\mc{O}\left(\left(M+\frac{Q^2}{\left(1-\sigma\right)^2}\right)\log\frac{1}{\epsilon}\right)$$
local component gradient computations. 
\end{theorem}

\textbf{Comparison with related work:}		
Table~\ref{tab:my-table} summarizes the rate comparison with recent related work, where, for the simplicity of presentation, we assume that all nodes have the same number of local functions, i.e., $M=m=\widetilde{m}$. It can be observed that in large-scale scenarios where~$\widetilde{m}$ is very large, both~\textbf{\texttt{GT-SVRG}} and~\textbf{\texttt{GT-SAGA}} improve upon the convergence rate of these methods in terms of the joint dependence on~$Q$ and~$\widetilde{m}$. We acknowledge that DSBA~\cite{DSBA} and ADFS~\cite{ADFS} achieve better iteration complexity than~\textbf{\texttt{GT-SVRG}} and~\textbf{\texttt{GT-SAGA}}, however, at the expense of computing the proximal mapping of a component function at each iteration.
	Although the computation of this proximal mapping is efficient for certain function classes, it can be very expensive for general functions. Finally, it is worth noting that all existing variance-reduced decentralized stochastic methods~\cite{DSA,DAVRG,edge_DSA,DSBA,ADFS} require symmetric weight matrices and thus undirected networks. In contrast,~\textbf{\texttt{GT-SVRG}} and~\textbf{\texttt{GT-SVRG}} only require doubly-stochastic weights and therefore can be implemented over certain classes of directed graphs that admit doubly-stochastic weights~\cite{weight_balance_digraph}. This provides more flexibility in topology design of the network.
\begin{center}
		\begin{table}[!h]
			\caption{Comparison of several state-of-the-art decentralized optimization methods}
			\centering
			\begin{tabular}{|c|c|}
				\hline
				\textbf{Algorithm} & \textbf{Convergence Rate} \\ \hline
				Gradient Tracking~\cite{harnessing} & $\mc{O}\left(\frac{\widetilde{m}Q^2}{\left(1-\sigma\right)^2}\log\frac{1}{\epsilon}\right)$  \\ \hline
				Gradient Tracking with Nesterov acceleration (see Theorem 3 in~\cite{AGT}) & $\mc{O}\left(\frac{\widetilde{m}Q^{\frac{5}{7}}}{\sigma^{1.5}\left(1-\sigma\right)^{1.5}}\log\frac{1}{\epsilon}\right)$  \\ \hline
				DSA~\cite{DSA} & $\mc{O}\left(\max\left\{\widetilde{m}Q,\frac{Q^4}{1-\sigma},\frac{1}{(1-\sigma)^2}\right\}\log\frac{1}{\epsilon}\right)$\\ \hline
				Edge-based DSA~\cite{edge_DSA} & linear (no explicit rate provided in terms of~$\widetilde{m},Q,\sigma$) \\ \hline
				Diffusion-AVRG~\cite{DAVRG} & linear (no explicit rate provided in terms of~$\widetilde{m},Q,\sigma$) \\ \hline
				\textbf{\texttt{GT-SAGA} (Part I~\cite{GT_SAGA})} &  $\mc{O}\left(\max\left\{\widetilde{m},\frac{Q^2}{\left(1-\sigma\right)^2}\right\}\log\frac{1}{\epsilon}\right)$ \\ \hline
				\textbf{\texttt{GT-SVRG} (this work)} &  $\mc{O}\left(\left(\widetilde{m}+\frac{Q^2}{\left(1-\sigma\right)^2}\right)\log\frac{1}{\epsilon}\right)$ \\ \hline
			\end{tabular}
			\label{tab:my-table}
		\end{table}
	\end{center}

	\textbf{Comparison with~\texttt{GT-SAGA}}: Recall from Part I~\cite{GT_SAGA} of this work that~\textbf{\texttt{GT-SAGA}} achieves~$\epsilon$-accuracy with $$\mc{O}\left(\max\left\{M,\frac{M}{m}\frac{Q^2}{\left(1-\sigma\right)^2}\right\}\log\frac{1}{\epsilon}\right)$$
	local component gradient computations. It can be observed that when data samples are distributed over the network in a highly unbalanced way, i.e.,~$M/m\gg1$,~\textbf{\texttt{GT-SVRG}} achieves better iteration complexity than~\textbf{\texttt{GT-SAGA}}. However, from a practical implementation point of view,
	an unbalanced data distribution may lead to a longer computation time in~\textbf{\texttt{GT-SVRG}}. This is due to the number of local gradient computations required at the end of each inner loop especially for nodes with large number of data samples. Furthermore,~\textbf{\texttt{GT-SVRG}} cannot execute the next inner loop before all nodes finish the local full gradient computation, leading to an overall increase in runtime. Clearly, there is an inherent trade-off between network synchrony and the storage of gradients as far as the relative implementation complexities of \textbf{\texttt{GT-SVRG}} and~\textbf{\texttt{GT-SAGA}} are concerned. If each each node is capable of storing all local component gradients, then~\textbf{\texttt{GT-SAGA}} may be preferred due to flexibility of implementation. On the other hand, for large-scale optimization problems where each node possesses a very large number of data samples, storing all component gradients may be infeasible and~\textbf{\texttt{GT-SVRG}} may be preferred. 
	
	In the next section, we present the convergence analysis of \texttt{\textbf{GT-SVRG}}.

\section{\textbf{\texttt{GT-SVRG}}: Convergence Analysis}
 In the rest of the paper, we assume~$p=1$ for the sake of simplicity. It is straightforward to develop the general case of~$p>1$ with the help of the Kronecker products; see e.g., the procedure in~\cite{AB}. For the purposes of analysis, we now write~\textbf{\texttt{GT-SVRG}} in the following matrix form,~$\forall t\geq0$ and~$0\leq k\leq K-1$, 
	\begin{subequations}\label{WWSVRG}
	\begin{align}
	\mb{x}^{t,k+1} &= W\mb{x}^{t,k} - \alpha\mb{y}^{t,k}, \label{svrg_av}\\
	\mb{y}^{t,k+1} &= W\mb{y}^{t,k} +
	\mb{v}^{t,k+1}-\mb{v}^{t,k}, \label{svrg_bv}
	\end{align}
\end{subequations}
where we use the following notation:
\begin{align*}
\mb{x}^{t,k}\triangleq\left[{\mb{x}^{t,k}_{1}}^\top,\cdots,{\mb{x}^{t,k}_{n}}^\top\right]^\top, \quad
\mb{y}^{t,k}\triangleq\left[{\mb{y}^{t,k}_{1}}^\top,\cdots,{\mb{y}^{t,k}_{n}}^\top\right]^\top, \quad
\mb{v}^{t,k}\triangleq\left[{\mb{v}^{t,k}_{1}}^\top,\cdots,{\mb{v}^{t,k}_{n}}^\top\right]^\top.
\end{align*}
We also define the following quantities:
\begin{align}
\ol{\mb{x}}^{t,k} &\triangleq \frac{1}{n}\mb{1}_n^\top\mb{x}^{t,k},\qquad\ol{\mb{y}}^{t,k} \triangleq \frac{1}{n}\mb{1}_n^\top\mb{y}^{t,k},\qquad \ol{\mb{v}}^{t,k}\triangleq \frac{1}{n}\mb{1}_n^\top\mb{v}^{t,k},
\nonumber
\\\nabla\mb{f}(\mb{x}^{t,k})&\triangleq[\nabla f_1(\mb{x}_1^{t,k})^\top,\dots,\nabla f_n(\mb{x}_n^{t,k})^\top]^\top,\qquad\mb{h}(\mb{x}^{t,k})\triangleq\frac{1}{n}\mb{1}_n^\top\nabla\mb{f}(\mb{x}^{t,k}).
\nonumber
\end{align}

\subsection{Preliminaries}
We denote~$\mc{F}^{t,k}$ as the~$\sigma$-algebra generated by the random variables up to the~$k$th-inner iteration of $t$th-outer loop,~i.e.,~$\{s_{i}^{l,r}\}_{i\in\mc{V}}^{l\leq t, r\leq k-1}.$
It is straightforward to observe that each local SVRG gradient~$\mb{v}_i^{k,t}$ is an unbiased estimator of the full local gradient~$\nabla f_i(\mb{x}_i^{t,k})$ given~$\mc{F}^{t,k}$, i.e.,
$$
\mathbb{E}\left[\mb{v}_i^{t,k}|\mc{F}^{t,k}\right]
= 
\mathbb{E}\left[\nabla f_{i,s_i^{t,k}}(\mb{x}_{i}^{t,k}) - \nabla f_{i,s_i^{t,k}}(\mb{x}_i^{t,0}) + \widetilde{\nabla}_i^t\Big|\mc{F}^{t,k}\right]
= \nabla f_i(\mb{x}_i^{t,k}).
$$
We first note that the average of gradient trackers preserves the average of local SVRG gradients. 
\begin{lem}\label{track}
$\forall t\geq0$ and~$0\leq k\leq K$,~$\ol{\mb{y}}^{t,k} = \ol{\mb{v}}^{t,k}$.
\end{lem}	
\begin{proof} Multiplying~$\frac{1}{n}\mb{1}_n^\top$ to~\eqref{svrg_bv}, we have:~$\forall t\geq0$ and~$0\leq k\leq K-1$,
\begin{align*}
\ol{\mb{y}}^{t,k+1} = \ol{\mb{y}}^{t,k}
+ \ol{\mb{v}}^{t,k+1} - \ol{\mb{v}}^{t,k}.
\end{align*} 
If~$t=0$, we have that for~$ 0\leq k\leq K-1$,
$$
\ol{\mb{y}}^{0,k+1} = \ol{\mb{y}}^{0,k}
+ \ol{\mb{v}}^{0,k+1} - \ol{\mb{v}}^{0,k}
= \ol{\mb{y}}^{0,k-1}
 - \ol{\mb{v}}^{0,k-1}
+ \ol{\mb{v}}^{0,k+1}
= \ol{\mb{v}}^{0,k+1},
$$
where in the last equality we used the initial condition that~$\mb{y}^{0,0} = \mb{v}^{0,0}$. Therefore,~$\mb{y}^{0,k} = \mb{v}^{0,k}, 0\leq k\leq K$. Now suppose that~$\ol{\mb{y}}^{l,k} = \ol{\mb{v}}^{l,k}$ for some~$l\geq0$ and~$0\leq k\leq K$. We have the following:
\begin{align*}
\ol{\mb{y}}^{l+1,k+1} = \ol{\mb{y}}^{l+1,k}
+ \ol{\mb{v}}^{l+1,k+1} - \ol{\mb{v}}^{l+1,k}
= \ol{\mb{y}}^{l+1,0} - \ol{\mb{v}}^{l+1,0}
+ \ol{\mb{v}}^{l+1,k+1} 
= \ol{\mb{y}}^{l,K} - \ol{\mb{v}}^{l,K}
+ \ol{\mb{v}}^{l+1,k+1} 
= \ol{\mb{v}}^{l+1,k+1}
.
\end{align*} 
Therefore,~$\mb{y}^{l+1,k} = \mb{v}^{l+1,k}, 0\leq k\leq K$. The proof follows by mathematical induction.
\end{proof}
Next, we present some standard lemmas in the context of stochastic gradient tracking methods and SVRG. Their proofs can be found in, for example,~\cite{harnessing,DSGT_Pu,DSGT_Xin,SVRG,DIGing}. Based on Lemma~\ref{track} and the fact that~$\mb{v}_i^{t,k}$ is an unbiased estimator of~$\nabla f_i(\mb{x}_i^{t,k})$, the following lemma is straightforward.
\begin{lem}\label{p3}
	$\mathbb{E}\left[\ol{\mb{y}}^{t,k}|\mc{F}^{t,k}\right] = \mathbb{E}\left[\ol{\mb{v}}^{t,k}|\mc{F}^{t,k}\right] = \mb{h}(\mb{x}^{t,k})$,~$\forall t\geq0$ and~$0\leq k\leq K$.
\end{lem}
The difference of~$\mb{h}(\mb{x}^{t,k})$ and~$\nabla f(\ol{\mb{x}}^{t,k})$ is bounded by the consensus error~$\left\|\mb{x}^{t,k}-\mb{1}_n\ol{\mb{x}}^{t,k}\right\|$ as follows. 
\begin{lem}\label{p4}
	$\left\|\mb{h}(\mb{x}^{t,k})-\nabla f(\ol{\mb{x}}^{t,k})\right\|\leq\frac{L}{\sqrt{n}}\left\|\mb{x}^{t,k}-\mb{1}_n\ol{\mb{x}}^{t,k}\right\|$,~$\forall t\geq0$ and~$0\leq k\leq K$.
\end{lem}
The weight matrix~$W$ is a contraction operator.
\begin{lem}
	$\forall \mb{x}\in\mathbb{R}^n, \left\|W\mb{x} - W_\infty\mb{x}\right\|\leq\sigma\left\|\mb{x} - W_\infty\mb{x}\right\|$, where~$W_\infty=\frac{\mb{1}_n\mb{1}_n^\top}{n}$.
\end{lem}
Descending along the direction of full gradient leads to a contraction in the optimality gap~\cite{nesterov_book}.
\begin{lem}
	Let~$f$ be~$\mu$-strongly-convex and~$L$-smooth. If~$0<\alpha\leq\frac{1}{L}$, the following holds, for~$\forall\mb{x}\in\mathbb{R}^p$,
	\begin{align*}
	\left\|\mb{x}-\alpha\nabla f(\mb{x}) -\mb{x}^*\right\|
	\leq(1-\mu\alpha)\left\|\mb{x}-\mb{x}^*\right\|
	\end{align*}
\end{lem}

\subsection{Auxiliary Results}
In order to develop the results, we first consider the progress made by one inner-loop iteration of~\textbf{\texttt{GT-SVRG}}. First, following~\cite{DSGT_Pu,DSGT_Xin}, we derive a contraction + perturbation bound for~$\mathbb{E}\left[\left\|\mb{x}^{t,k}-\mb{1}_n\ol{\mb{x}}^{t,k}\right\|^2\right]$.  
\begin{lem}\label{consensus_er_svrg}
The following holds:~$\forall t\geq0$ and~$0\leq k\leq K-1$,
	$$\mathbb{E}\left[\left\|\mb{x}^{t,k+1}-\mb{1}_n\ol{\mb{x}}^{t,k+1}\right\|^2\right]\leq
	\frac{1+\sigma^2}{2}\mathbb{E}\left[\left\|\mb{x}^{t,k}-\mb{1}_n\ol{\mb{x}}^{t,k}\right\|^2\right] + \frac{2\alpha^2}{1-\sigma^2}\mathbb{E}\left[\left\|\mb{y}^{t,k}-\mb{1}_n\ol{\mb{y}}^{t,k}\right\|^2\right].$$
\end{lem}
	\begin{proof}
	Following from~\eqref{svrg_av}, we have
	\begin{align*}
	&\left\|\mb{x}^{t,k+1}-\mb{1}_n\ol{\mb{x}}^{t,k+1}\right\|^2 = \left\|W\mb{x}^{t,k} - \alpha\mb{y}^{t,k}-W_\infty\left(W\mb{x}^{t,k} - \alpha\mb{y}^{t,k}\right)\right\|^2 \\
	=& \left\|W\mb{x}^{t,k} - W_\infty\mb{x}^{t,k}\right\|^2 + \alpha^2\left\|\mb{y}^{t,k} - W_\infty\mb{y}^{t,k}\right\|^2
	- 2\alpha\Big\langle W\mb{x}^{t,k} -W_\infty\mb{x}^{t,k}, \mb{y}^{t,k} -W_\infty\mb{y}^{t,k} \Big\rangle \\
	=&~\sigma^2\left\|\mb{x}^{t,k} - W_\infty\mb{x}^{t,k}\right\|^2 + \alpha^2\left\|\mb{y}^{t,k} - W_\infty\mb{y}^{t,k}\right\|^2
	+2\sigma\left\|\mb{x}^{t,k} - W_\infty\mb{x}^{t,k}\right\|\alpha\left\|\mb{y}^{t,k} - W_\infty\mb{y}^{t,k}\right\| \\
	\leq&~\sigma^2\left\|\mb{x}^{t,k} - W_\infty\mb{x}^{t,k}\right\|^2 + \alpha^2\left\|\mb{y}^{t,k} - W_\infty\mb{y}^{t,k}\right\|^2 \nonumber\\
	&+ \sigma\left(\frac{1-\sigma^2}{2\sigma}\left\|\mb{x}^{t,k} - W_\infty\mb{x}^{t,k}\right\|^2
	+ \frac{2\sigma}{1-\sigma^2}\alpha^2\left\|\mb{y}^{t,k} - W_\infty\mb{y}^{t,k}\right\|^2\right),\\
	=&~\frac{1+\sigma^2}{2}\left\|\mb{x}^{t,k} - W_\infty\mb{x}^{t,k}\right\|^2+\alpha^2\left(1+\frac{2\sigma^2}{1-\sigma^2}\right)\left\|\mb{y}^{t,k} - W_\infty\mb{y}^{t,k}\right\|^2
	\end{align*}
	and the proof follows from~$1+\sigma^2<2$ and taking the total expectation.
\end{proof}
Next, we bound the optimality gap~$\mathbb{E}\left[\left\|\ol{\mb x}^{t,k}-\mb{x}^*\right\|^2 | \mc{F}^{t,k}\right]$, following the procedure in~\cite{DSGT_Pu,DSGT_Xin}. 
\begin{lem}\label{opt_exp_svrg}
The following holds:~$\forall t\geq0$ and~$0\leq k\leq K-1$,
	\begin{align*}
	\mathbb{E}\left[\left\|\ol{\mb x}^{t,k+1}-\mb{x}^*\right\|^2 | \mc{F}^{t,k}\right] 
	=& \left\|\ol{\mb x}^{t,k}-\alpha\nabla f(\ol{\mb x}^{t,k})-\mb{x}^* \right\|^2 + 2\alpha\Big\langle \ol{\mb x}^{t,k}-\alpha\nabla f(\ol{\mb x}^{t,k})-\mb{x}^*, \nabla f(\ol{\mb x}^{t,k})-\mb{h}(\mb{x}^{t,k})\Big\rangle
	\nonumber\\
	&+ \alpha^2\left\|\nabla f(\ol{\mb x}^{t,k})
	-\mb{h}(\mb{x}^{t,k})\right\|^2
	+ \frac{\alpha^2}{n^2}\mathbb{E}\left[\left\|
	\mb{v}^{t,k}-\nabla\mb{f}(\mb{x}^{t,k})\right\|^2 
	\Big| \mc{F}^{t,k}\right]
	.
	\end{align*}
\end{lem}
\begin{proof}
Multiplying~$\frac{1}{n}\mb{1}_n^\top$ to~\eqref{svrg_av} obtains~$\ol{\mb{x}}^{t,k+1}=\ol{\mb{x}}^{t,k}-\alpha\ol{\mb{y}}^{t,k},\forall t\geq 0~\text{and}~0\leq k \leq K-1.$ Then we have,
	\begin{align*}
	&\left\|\ol{\mb x}^{t,k+1}-\mb{x}^*\right\|^2 = \left\|\ol{\mb x}^{t,k}-\alpha\ol{\mb y}^{t,k}-\mb{x}^*\right\|^2\\
	=& \left\|\ol{\mb x}^{t,k}-\alpha\nabla f(\ol{\mb x}^{t,k})-\mb{x}^*+\alpha\left(\nabla f(\ol{\mb x}^{t,k})-\ol{\mb y}^{t,k}\right)\right\|^2 \\
	=& \left\|\ol{\mb x}^{t,k}-\alpha\nabla f(\ol{\mb x}^{t,k})-\mb{x}^* \right\|^2 + 2\alpha\Big\langle \ol{\mb x}^{t,k}-\alpha\nabla f(\ol{\mb x}^{t,k})-\mb{x}^*, \nabla f(\ol{\mb x}^{t,k})-\ol{\mb y}^{t,k}\Big\rangle + \alpha^2 \left\|\nabla f(\ol{\mb x}^{t,k})-\ol{\mb y}^{t,k}\right\|^2. 
	\end{align*}
	Recall that~$\mathbb{E}\left[\ol{\mb{y}}^{t,k}|\mc{F}^{t,k}\right] = \mb{h}(\mb{x}^{t,k})$ from Lemma~\ref{p3}. We take the expectation from bothsides given~$\mc{F}^{t,k}$ to obtain:
	\begin{align}\label{d1}
	\mathbb{E}\left[\left\|\ol{\mb x}^{t,k+1}-\mb{x}^*\right\|^2 | \mc{F}^{t,k}\right] 
	=& \left\|\ol{\mb x}^{t,k}-\alpha\nabla f(\ol{\mb x}^{t,k})-\mb{x}^* \right\|^2 + 2\alpha\Big\langle \ol{\mb x}^{t,k}-\alpha\nabla f(\ol{\mb x}^{t,k})-\mb{x}^*, \nabla f(\ol{\mb x}^{t,k})-\mb{h}(\mb{x}^{t,k})\Big\rangle
	\nonumber\\
	&+ \alpha^2 \mathbb{E}\left[\left\|\nabla f(\ol{\mb x}^{t,k})-\ol{\mb y}^{t,k}\right\|^2 | \mc{F}^{t,k}\right].
	\end{align}
	We split the last term above~$\left\|\nabla f(\ol{\mb x}^{t,k})-\ol{\mb y}^{t,k}\right\|^2$ as consensus error + variance as follows. 
	\begin{align}\label{e1}
	&\mathbb{E}\left[\left\|\nabla f(\ol{\mb x}^{t,k})-\ol{\mb y}^{t,k}\right\|^2 | \mc{F}^{k}\right]\nonumber\\
	=& \mathbb{E}\left[\left\|\nabla f(\ol{\mb x}^{t,k})
	-\mb{h}(\mb{x}^{t,k}) + \mb{h}(\mb{x}^{t,k})
	-\ol{\mb y}^{t,k}\right\|^2 | \mc{F}^{t,k}\right] \nonumber\\
	=& \underbrace{\left\|\nabla f(\ol{\mb x}^{t,k})
		-\mb{h}(\mb{x}^{t,k})\right\|^2 }_{\text{consensus error}} 
	+ \underbrace{\mathbb{E}\left[\left\|\mb{h}(\mb{x}^{t,k})
		-\ol{\mb y}^{t,k}\right\|^2 \big| \mc{F}^{t,k}\right]}_{\text{variance}}
	+ \underbrace{2\Big\langle \nabla f(\ol{\mb x}^{t,k})
		-\mb{h}(\mb{x}^{t,k}), \mathbb{E}\left[\mb{h}(\mb{x}^{t,k})
		-\ol{\mb y}^{t,k}\big| \mc{F}^{t,k} \right]\Big\rangle}_{=0} 
	\end{align}
	The variance term can be simplified as follows:
	\begin{align}\label{v1}
	&~\mathbb{E}\left[\left\|\mb{h}(\mb{x}^{t,k})
	-\ol{\mb y}^{t,k}\right\|^2 \Big| \mc{F}^{t,k}\right] \nonumber\\
	=&~\mathbb{E}\left[\left\|
	\frac{1}{n}\sum_{i=1}^{n}\left( \nabla f_i(\mb{x}_i^{t,k})-\mb{v}_i^{t,k}\right)
	\right\|^2 \Bigg| \mc{F}^{t,k}\right] 
	= \frac{1}{n^2}\mathbb{E}\left[\left\|
	\sum_{i=1}^{n}\left(\nabla f_i(\mb{x}_i^{t,k})-\mb{v}_i^{t,k}\right)
	\right\|^2 \Bigg| \mc{F}^{t,k}\right] \nonumber\\
	=&~\frac{1}{n^2}\mathbb{E}\left[\sum_{i=1}^{n}\left\|
	\nabla f_i(\mb{x}_i^{t,k})-\mb{v}_i^{t,k}\right\|^2 
	+\sum_{i\neq j}\Big\langle \nabla f_i(\mb{x}_i^{t,k})-\mb{v}_i^{t,k}, \nabla f_j(\mb{x}_j^k)-\mb{v}_j^k\Big\rangle 
	\Bigg| \mc{F}^{t,k}\right] \nonumber\\
	=&~\frac{1}{n^2}\sum_{i=1}^{n}\mathbb{E}\left[\left\|
	\nabla f_i(\mb{x}_i^{t,k})-\mb{v}_i^{t,k}\right\|^2 
	\Big| \mc{F}^{t,k}\right]
	= \frac{1}{n^2}\mathbb{E}\left[\left\|
	\nabla\mb{f}(\mb{x}^{t,k})-\mb{v}^{t,k}\right\|^2 
	\Big| \mc{F}^{t,k}\right]
	,	
	\end{align}                                                   
	where the second last equality uses the fact that~$\{\mb{v}_i^k\}$ are independent with each other given~$\mc{F}^{t,k}$. The proof follows from using~\eqref{e1} and~\eqref{v1} in~\eqref{d1}.
\end{proof}

Next, we bound the gradient variance term~$\mathbb{E}\left[\left\| \mb{v}^{t,k}-\nabla\mb{f}(\mb{x}^{t,k})\right\|^2 
\Big| \mc{F}^{t,k}\right]$, following a similar procedure in~\cite{SVRG}.
\begin{lem}\label{var_svrg}
The following holds:~$\forall t\geq0$ and~$0\leq k\leq K$,
\begin{align*}
&\mathbb{E}\left[\left\| \mb{v}^{t,k}-\nabla \mb{f}(\mb{x}^{t,k})\right\|^2 
\Big| \mc{F}^{t,k}\right] \nonumber\\
&\leq
4L^2\left\|\mb{x}^{t,k}-\mb{1}_n\ol{\mb{x}}^{t,k}\right\|^2
+ 4nL^2\left\|\ol{\mb{x}}^{t,k}-\mb{x}^*\right\|^2
+ 4L^2\left\|\mb{x}^{t,0}-\mb{1}_n\ol{\mb{x}}^{t,0}\right\|^2
+ 4nL^2\left\|\ol{\mb{x}}^{t,0}-\mb{x}^*\right\|^2,
\end{align*}
\end{lem}
\begin{proof}
By the local SVRG update and using standard variance decomposition, we have that:
\begin{align}
&\mathbb{E}\left[\left\| \mb{v}_i^{t,k}-\nabla f_i(\mb{x}_i^{t,k})\right\|^2 
\Big| \mc{F}^{t,k}\right] \nonumber\\
=&~
\mathbb{E}\Big[\Big\|\underbrace{\nabla f_{i,s_i^{t,k}}(\mb{x}_{i}^{t,k}) - \nabla f_{i,s_i^{t,k}}(\mb{x}_i^{t,0})}_{X^{t,k}} -\underbrace{\left( \nabla f_i(\mb{x}_i^{t,k})-\widetilde{\nabla}_i^t\right)}_{\mathbb{E}\left[X^{t,k}|\mc{F}^{t,k}\right]}\Big\|^2\Big|\mc{F}^{t,k}\Big] \nonumber\\
\leq&~
\mathbb{E}\left[\left\|\nabla f_{i,s_i^{t,k}}(\mb{x}_{i}^{t,k}) - \nabla f_{i,s_i^{t,k}}(\mb{x}_i^{t,0})\right\|^2\Big|\mc{F}^{t,k}\right] \nonumber\\
\leq&~
2\mathbb{E}\left[\left\|\nabla f_{i,s_i^{t,k}}(\mb{x}_{i}^{t,k}) - \nabla f_{i,s_i^{t,k}}(\mb{x}^*)\right\|^2\Big|\mc{F}^{t,k}\right]
+
2\mathbb{E}\left[\left\|\nabla f_{i,s_i^{t,k}}(\mb{x}_{i}^{t,0}) - \nabla f_{i,s_i^{t,k}}(\mb{x}^{*})\right\|^2\Big|\mc{F}^{t,k}\right] \nonumber \\
=&~ 
\frac{2}{m_i}\sum_{j=1}^{m_i}\left\|\nabla f_{i,j}(\mb{x}_{i}^{t,k}) - \nabla f_{i,j}(\mb{x}^*)\right\|^2
+
\frac{2}{m_i}\sum_{j=1}^{m_i}\left\|\nabla f_{i,j}(\mb{x}_{i}^{t,0}) - \nabla f_{i,j}(\mb{x}^*)\right\|^2 \nonumber\\
\leq&~
2L^2\left\|\mb{x}_i^{t,k}-\mb{x}^*\right\|^2
+ 2L^2\left\|\mb{x}_i^{t,0}-\mb{x}^*\right\|^2
\nonumber\\
\leq&~
4L^2\left\|\mb{x}_i^{t,k}-\ol{\mb{x}}^{t,k}\right\|^2
+ 4L^2\left\|\ol{\mb{x}}^{t,k}-\mb{x}^*\right\|^2
+ 4L^2\left\|\mb{x}_i^{t,0}-\ol{\mb{x}}^{t,0}\right\|^2
+ 4L^2\left\|\ol{\mb{x}}^{t,0}-\mb{x}^*\right\|^2.
\end{align}
Summing the above inequality over~$i$ from~$1$ to~$n$ completes the proof.
\end{proof}
Next, we use Lemma~\ref{var_svrg} to refine the optimality bound in Lemma~\ref{opt_exp_svrg}.

\begin{lem}\label{opt_er_svrg}
If~$0<\alpha\leq\frac{\mu}{8L^2}$, the following holds:~$\forall t\geq0$ and~$0\leq k\leq K-1$,
\begin{align*}
\mathbb{E}\left[\left\|\ol{\mb x}^{t,k+1}-\mb{x}^*\right\|^2 \right]
\leq&~
\left(1-\frac{\mu\alpha}{2}\right)\mathbb{E}\left[\left\|\ol{\mb x}^{t,k}-\mb{x}^*\right\|^2\right]
+\frac{3L^2\alpha }{2\mu n}\mathbb{E}\left[\left\|\mb{x}^{t,k}-\mb{1}_n\ol{\mb{x}}^{t,k}\right\|^2\right]\nonumber\\
&+\frac{4L^2\alpha^2}{n^2} \mathbb{E}\left[\left\|\mb{x}^{t,0}-\mb{1}_n\ol{\mb{x}}^{t,0}\right\|^2\right]
+\frac{4L^2\alpha^2}{n} \mathbb{E}\left[\left\|\ol{\mb{x}}^{t,0}-\mb{x}^*\right\|^2\right].
\end{align*}

\end{lem}	
\begin{proof}
Recall Lemma~\ref{opt_exp_svrg} and use standard contraction in gradient descent. 
\begin{align*}
&\mathbb{E}\left[\left\|\ol{\mb x}^{t,k+1}-\mb{x}^*\right\|^2 | \mc{F}^{t,k}\right] \nonumber\\
\leq&~
\left(1-\mu\alpha\right)^2\left\|\ol{\mb x}^{t,k}-\mb{x}^*\right\|^2
+ 2\alpha\left(1-\mu\alpha\right)\left\|\ol{\mb x}^{t,k}-\mb{x}^*\right\|\left\|\nabla f(\ol{\mb x}^{k})
-\mb{h}(\mb{x}^{t,k})\right\|+\alpha^2\left\|\nabla f(\ol{\mb x}^{t,k})
-\mb{h}(\mb{x}^{t,k})\right\|^2 \nonumber\\
&+\frac{\alpha^2}{n^2}\mathbb{E}\left[\left\| \mb{v}^{t,k}-\nabla\mb{f}(\mb{x}^{t,k})\right\|^2 
\Big| \mc{F}^{k}\right] \nonumber\\
\leq&~
\left(1-\mu\alpha\right)^2\left\|\ol{\mb x}^{t,k}-\mb{x}^*\right\|^2
+ \alpha\left(1-\mu\alpha\right)\left(\mu\left\|\ol{\mb x}^{t,k}-\mb{x}^*\right\|^2+\frac{1}{\mu}\left\|\nabla f(\ol{\mb x}^{t,k})
-\mb{h}(\mb{x}^{t,k})\right\|^2\right) \nonumber\\
&+\alpha^2\left\|\nabla f(\ol{\mb x}^{t,k})
-\mb{h}(\mb{x}^{t,k})\right\|^2 +\frac{\alpha^2}{n^2}\mathbb{E}\left[\left\| \mb{v}^{t,k}-\nabla\mb{f}(\mb{x}^{t,k})\right\|^2 
\Big| \mc{F}^{t,k}\right] \nonumber\\
=&~
\left(1-\mu\alpha\right)\left\|\ol{\mb x}^{t,k}-\mb{x}^*\right\|^2
+\frac{\alpha}{\mu}\left\|\nabla f(\ol{\mb x}^{t,k})
-\mb{h}(\mb{x}^{t,k})\right\|^2 +\frac{\alpha^2}{n^2}\mathbb{E}\left[\left\| \mb{g}^{t,k}-\nabla\mb{f}(\mb{x}^{t,k})\right\|^2 
\Big| \mc{F}^{t,k}\right]. \nonumber
\end{align*}
We then use Lemma~\ref{p4} and Lemma~\ref{var_svrg} to simplify the above inequality as follows:
\begin{align*}
&\mathbb{E}\left[\left\|\ol{\mb x}^{t,k+1}-\mb{x}^*\right\|^2 | \mc{F}^{t,k}\right]
\nonumber\\
\leq&
\left(1-\mu\alpha\right)\left\|\ol{\mb x}^{t,k}-\mb{x}^*\right\|^2
+\frac{\alpha L^2}{\mu n}\left\|\mb{x}^{t,k}-\mb{1}_n\ol{\mb{x}}^{t,k}\right\|^2 \nonumber\\
&+\frac{\alpha^2}{n^2}\left(4L^2\left\|\mb{x}^{t,k}-\mb{1}_n\ol{\mb{x}}^{t,k}\right\|^2
+ 4nL^2\left\|\ol{\mb{x}}^{t,k}-\mb{x}^*\right\|^2
+ 4L^2\left\|\mb{x}^{t,0}-\mb{1}_n\ol{\mb{x}}^{t,0}\right\|^2
+ 4nL^2\left\|\ol{\mb{x}}^{t,0}-\mb{x}^*\right\|^2\right). \nonumber\\
=&
\left(1-\mu\alpha+\frac{4L^2\alpha^2}{n}\right)\left\|\ol{\mb x}^{t,k}-\mb{x}^*\right\|^2
+\frac{\alpha L^2}{ n}\left(\frac{1}{\mu}+\frac{4\alpha}{n}\right)\left\|\mb{x}^{t,k}-\mb{1}_n\ol{\mb{x}}^{t,k}\right\|^2 \nonumber\\
&+\frac{4L^2\alpha^2}{n^2} \left\|\mb{x}^{t,0}-\mb{1}_n\ol{\mb{x}}^{t,0}\right\|^2
+\frac{4L^2\alpha^2}{n} \left\|\ol{\mb{x}}^{t,0}-\mb{x}^*\right\|^2.
\end{align*}
If~$\alpha\leq\frac{\mu n}{8L^2}$, then~$1-\mu\alpha+\frac{4L^2\alpha^2}{n}\leq1-\frac{\mu\alpha}{2}$ and~$\frac{\alpha L^2}{ n}\left(\frac{1}{\mu}+\frac{4\alpha}{n}\right)\leq\frac{3\alpha L^2}{2\mu n}$, which completes the proof.
\end{proof}
	
Next, we bound the gradient tracking error~$\mathbb{E}\left[\left\|\mb{y}^{t,k+1} - \mb{1}_n\ol{\mb{y}}^{t,k+1}\right\|^2\right]$.	

\begin{lem}\label{track_er_svrg}
	The following holds:~$\forall t\geq0$ and~$0\leq k\leq K-1$,
\begin{align*}
&\mathbb{E}\left[\left\|\mb{y}^{t,k+1} - \mb{1}_n\ol{\mb{y}}^{t,k+1}\right\|^2\right] \nonumber\\
\leq&~
\frac{98L^2}{1-\sigma^2}\mathbb{E}\left[\left\|\mb{x}^{t,k}-\mb{1}_n\ol{\mb{x}}^{t,k}\right\|^2\right] + \frac{66nL^2}{1-\sigma^2}\mathbb{E}\left[\left\|\ol{\mb x}^{t,k}-\mb{x}^*\right\|^2\right]
+ \left(\frac{1+\sigma^2}{2}+\frac{40L^2\alpha^2}{1-\sigma^2}\right)\mathbb{E}\left[\left\|\mb{y}^{t,k} - \mb{1}_n\ol{\mb{y}}^{t,k}\right\|^2\right] 
\nonumber\\
&+ \frac{58L^2}{1-\sigma^2}\mathbb{E}\left[\left\|\mb{x}^{t,0}-\mb{1}_n\ol{\mb{x}}^{t,0}\right\|^2\right]
+ \frac{58nL^2}{1-\sigma^2}\mathbb{E}\left[\left\|\ol{\mb{x}}^{t,0}-\mb{x}^*\right\|^2\right].
\end{align*}
\end{lem}
\begin{proof}
Using the gradient tracking update, we have:
\begin{align*}
	&\left\|\mb{y}^{t,k+1} - W_{\infty}\mb{y}^{t,k+1}\right\|^2
	\nonumber\\
	=&
	\left\|W\mb{y}^{t,k} +
	\mb{v}^{t,k+1}-\mb{v}^{t,k}-
	W_{\infty}\left(W\mb{y}^{t,k} +
	\mb{v}^{t,k+1}-\mb{v}^{t,k}\right)
	\right\|^2  \nonumber\\ 
	=&
	\left\|W\mb{y}^{t,k}-
	W_{\infty}\mb{y}^{t,k} + \left(I_n-
	W_{\infty}\right)\left(
	\mb{v}^{t,k+1}-\mb{v}^{t,k}\right)
	\right\|^2 \nonumber\\
	\leq&~
	\sigma^2 \left\|\mb{y}^{t,k}-
	W_{\infty}\mb{y}^{t,k}\right\|^2 + \left\|\mb{v}^{t,k+1}-\mb{v}^{t,k}\right\|^2
	+ 2\Big\langle W\mb{y}^{t,k}-
	W_{\infty}\mb{y}^{t,k}, 
	\mb{v}^{t,k+1}-\mb{v}^{t,k}\Big\rangle,
	\end{align*}
where we used the fact that~$\left\|I-W_\infty\right\| = 1$ and~$\big\langle W\mb{y}^{t,k}-
W_{\infty}\mb{y}^{t,k}, 
W_\infty\left(\mb{v}^{t,k+1}-\mb{v}^{t,k}\right)\big\rangle = 0$.

Next, we bound~$\mathbb{E}\left[\left\|\mb{v}^{t,k+1}-\mb{v}^{t,k}\right\|^2\Big|\mc{F}^{t,k}\right]$.
\begin{align}
&\mathbb{E}\left[\left\|\mb{v}^{t,k+1}-\mb{v}^{t,k}\right\|^2\Big|\mc{F}^{t,k}\right] \nonumber\\
\leq&~
2\mathbb{E}\left[\left\|\mb{v}^{t,k+1}-\mb{v}^{t,k} 
-\left(\nabla\mb{f}(\mb{x}^{t,k+1})-\nabla\mb{f}(\mb{x}^{t,k})\right)\right\|^2\Big|\mc{F}^{t,k}\right]
+
2\mathbb{E}\left[\left\|\nabla\mb{f}(\mb{x}^{t,k+1})-\nabla\mb{f}(\mb{x}^{t,k})\right\|^2\Big|\mc{F}^{t,k}\right] 
\nonumber\\
\triangleq&~2U_1 + 2U_2.
\end{align}
We start with~$U_2$. Following~\cite{DSGT_Xin},
\begin{align}\label{gradf}
\left\|\nabla\mb{f}(\mb{x}^{t,k+1})-\nabla\mb{f}(\mb{x}^{t,k})\right\|^2
\leq&~L^2\left\|\mb{x}^{t,k+1}-\mb{x}^{t,k}\right\|^2
= L^2\left\|W\mb{x}^{t,k}-\alpha\mb{y}^{t,k}-\mb{x}^{t,k}\right\|^2
\nonumber\\
=&~ L^2\left\|\left(W-I_n\right)\left(\mb{x}^{t,k}-W_\infty\mb{x}^{t,k}\right)-\alpha\mb{y}^{t,k}\right\|^2
\nonumber\\
\leq&~
8L^2\left\|\mb{x}^{t,k}-W_\infty\mb{x}^{t,k}\right\|^2 + 2\alpha^2L^2\left\|\mb{y}^{t,k}\right\|^2,
\end{align}
where~$\left\|\mb{y}^{t,k}\right\|^2$ can be bounded as the following:
\begin{align}
\left\|\mb{y}^{t,k}\right\|
=&~\left\|\mb{y}^{t,k} - W_\infty\mb{y}^{t,k} + W_\infty\mb{v}^{t,k}
-W_\infty\nabla\mb{f}(\mb{x}^{t,k}) + W_\infty\nabla\mb{f}(\mb{x}^{t,k})
-W_\infty\nabla\mb{f}(\mb{1}_n\mb{x}^*)\right\| \nonumber\\
\leq&~
\left\|\mb{y}^{t,k} - W_\infty\mb{y}^{t,k}\right\|
+ \sqrt{n} \left\|\ol{\mb{v}}^{t,k}-\mb{h}(\mb{x}^{t,k})\right\| 
+ L\left\|\mb{x}^{t,k}-\mb{1}_n\mb{x}^*\right\|
\nonumber\\
\leq&~
\left\|\mb{y}^{t,k} - W_\infty\mb{y}^{t,k}\right\|
+ L\left\|\mb{x}^{t,k}-\mb{1}_n\ol{\mb{x}}^{t,k}\right\|
+ \sqrt{n}L\left\|\ol{\mb{x}}^{t,k}-\mb{x}^*\right\|
+ \sqrt{n} \left\|\ol{\mb{v}}^{t,k}-\mb{h}(\mb{x}^{t,k})\right\|.
\nonumber
\end{align}
Squaring the above inequality, we have:
\begin{align}\label{y_svrg}
\left\|\mb{y}^{t,k}\right\|^2
\leq 4\left\|\mb{y}^{t,k} - \mb{1}_n\ol{\mb{y}}^{t,k}\right\|^2
+ 4L^2\left\|\mb{x}^{t,k}-\mb{1}_n\ol{\mb{x}}^{t,k}\right\|^2
+ 4nL^2\left\|\ol{\mb{x}}^{t,k}-\mb{x}^*\right\|^2
+ \frac{4}{n}\left\|\mb{v}^{t,k}-\nabla\mb{f}(\mb{x}^{t,k})\right\|^2.
\end{align}
Using~\eqref{gradf} and~\eqref{y_svrg} bounds~$U_2$ as follows:
\begin{align}
U_2\leq&
\left(8L^2 + 8L^4\alpha^2\right)\left\|\mb{x}^{t,k}-\mb{1}_n\ol{\mb{x}}^{t,k}\right\|^2
+ 8nL^4\alpha^2\left\|\ol{\mb{x}}^{t,k}-\mb{x}^*\right\|^2
+ 8L^2\alpha^2\mathbb{E}\left[\left\|\mb{y}^{t,k} - \mb{1}_n\ol{\mb{y}}^{t,k}\right\|^2\Big |\mc{F}^{t,k}\right] \nonumber\\
&+ \frac{8L^2\alpha^2}{n}\left(4L^2\left\|\mb{x}^{t,k}-\mb{1}_n\ol{\mb{x}}^{t,k}\right\|^2
+ 4nL^2\left\|\ol{\mb{x}}^{t,k}-\mb{x}^*\right\|^2
+ 4L^2\left\|\mb{x}^{t,0}-\mb{1}_n\ol{\mb{x}}^{t,0}\right\|^2
+ 4nL^2\left\|\ol{\mb{x}}^{t,0}-\mb{x}^*\right\|^2\right)
\nonumber
\end{align}
If~$\alpha\leq\frac{1}{2L}$, then~$\alpha^2\leq\frac{1}{4L^2}$, we have the following:
\begin{align}
U_2\leq&~
18L^2\left\|\mb{x}^{t,k}-\mb{1}_n\ol{\mb{x}}^{t,k}\right\|^2
+ 10nL^2\left\|\ol{\mb{x}}^{t,k}-\mb{x}^*\right\|^2
+ 8L^2\alpha^2\mathbb{E}\left[\left\|\mb{y}^{t,k} - \mb{1}_n\ol{\mb{y}}^{t,k}\right\|^2\Big |\mc{F}^{t,k}\right] \nonumber\\
&+ \frac{8L^2}{n}\left\|\mb{x}^{t,0}-\mb{1}_n\ol{\mb{x}}^{t,0}\right\|^2
+ 8L^2\left\|\ol{\mb{x}}^{t,0}-\mb{x}^*\right\|^2.
\end{align}
Next we bound~$U_1$. Following a similar argument as before,
\begin{align}\label{U1}
U_1 
=&~
\mathbb{E}\left[\left\|\mb{v}^{t,k+1}-\mb{v}^{t,k} 
-\left(\nabla\mb{f}(\mb{x}^{t,k+1})-\nabla\mb{f}(\mb{x}^{t,k})\right)\right\|^2\Big|\mc{F}^{t,k}\right]
\nonumber\\
=&~
\mathbb{E}\left[\left\|\mb{v}^{t,k+1}-\nabla\mb{f}(\mb{x}^{t,k+1})\right\|^2\Big|\mc{F}^{t,k}\right]
+\mathbb{E}\left[\left\|\mb{v}^{t,k}-\nabla\mb{f}(\mb{x}^{t,k})\right\|^2\Big|\mc{F}^{t,k}\right]
\nonumber\\
&+ \mathbb{E}\left[\Big\langle\mb{v}^{t,k+1}-\nabla\mb{f}(\mb{x}^{t,k+1}),\mb{v}^{t,k}-\nabla\mb{f}(\mb{x}^{t,k})\Big\rangle\Big|\mc{F}^{t,k}\right]
\nonumber\\
=&~
\mathbb{E}\left[\mathbb{E}\left[\left\|\mb{v}^{t,k+1}-\nabla\mb{f}(\mb{x}^{t,k+1})\right\|^2\Big|\mc{F}^{t,k+1}\right]\Big|\mc{F}^{t,k}\right]
+\mathbb{E}\left[\left\|\mb{v}^{t,k} 
-\nabla\mb{f}(\mb{x}^{t,k})\right\|^2\Big|\mc{F}^{t,k}\right].
\end{align}
We first bound~$\mathbb{E}\left[\mathbb{E}\left[\left\|\mb{v}^{t,k+1}-\nabla\mb{f}(\mb{x}^{t,k+1})\right\|^2\Big|\mc{F}^{t,k+1}\right]\Big|\mc{F}^{t,k}\right]$.
\begin{align}
&\mathbb{E}\left[\mathbb{E}\left[\left\|\mb{v}^{t,k+1}-\nabla\mb{f}(\mb{x}^{t,k+1})\right\|^2\Big|\mc{F}^{t,k+1}\right]\Big|\mc{F}^{t,k}\right] \nonumber\\
\leq&~ 4L^2\mathbb{E}\left[\left\|\mb{x}^{t,k+1}-\mb{1}_n\ol{\mb{x}}^{t,k+1}\right\|^2\Big|\mc{F}^{t,k}\right]
+ 4nL^2\mathbb{E}\left[\left\|\ol{\mb{x}}^{t,k+1} - \mb{x}^*\right\|^2\Big|\mc{F}^{t,k}\right]\nonumber\\
&+ 4L^2\left\|\mb{x}^{t,0}-\mb{1}_n\ol{\mb{x}}^{t,0}\right\|^2
+ 4nL^2\left\|\ol{\mb{x}}^{t,0}-\mb{x}^*\right\|^2
\nonumber\\
\leq&~
4L^2\left(\frac{1+\sigma^2}{2}\left\|\mb{x}^{t,k}-\mb{1}_n\ol{\mb{x}}^{t,k}\right\|^2 + \frac{2\alpha^2}{1-\sigma^2}\mathbb{E}\left[\left\|\mb{y}^{t,k}-\mb{1}_n\ol{\mb{y}}^{t,k}\right\|^2\Big |\mc{F}^{t,k}\right]\right) \nonumber\\
&+4nL^2\left(\left(1-\frac{\mu\alpha}{2}\right)\left\|\ol{\mb x}^{t,k}-\mb{x}^*\right\|^2
+ \frac{3L^2\alpha}{2\mu n}
\left\|\mb{x}^{t,k}-\mb{1}_n\ol{\mb{x}}^{t,k}\right\|^2
+ \frac{4L^2\alpha^2}{n^2} \left\|\mb{x}^{t,0}-\mb{1}_n\ol{\mb{x}}^{t,0}\right\|^2
+\frac{4L^2\alpha^2}{n} \left\|\ol{\mb{x}}^{t,0}-\mb{x}^*\right\|^2\right)\nonumber\\
&+ 4L^2\left\|\mb{x}^{t,0}-\mb{1}_n\ol{\mb{x}}^{t,0}\right\|^2
+ 4nL^2\left\|\ol{\mb{x}}^{t,0}-\mb{x}^*\right\|^2
\nonumber\\
\leq&~
5L^2\left\|\mb{x}^{t,k}-\mb{1}_n\ol{\mb{x}}^{t,k}\right\|^2 + 4nL^2\left\|\ol{\mb x}^{t,k}-\mb{x}^*\right\|^2
+ \frac{8L^2\alpha^2}{1-\sigma^2}\mathbb{E}\left[\left\|\mb{y}^{t,k}-\mb{1}_n\ol{\mb{y}}^{t,k}\right\|^2\Big|\mc{F}^{t,k} \right] \nonumber\\
&+ 5L^2\left\|\mb{x}^{t,0}-\mb{1}_n\ol{\mb{x}}^{t,0}\right\|^2
+ 5nL^2\left\|\ol{\mb{x}}^{t,0}-\mb{x}^*\right\|^2,
\end{align}
where in the last inequality we set~$0<\alpha\leq\frac{\mu}{6L^2}$.
Now we derive an upper bound for~$U_1$ as follows.
\begin{align}\label{U2}
U_1 \leq&~
5L^2\left\|\mb{x}^{t,k}-\mb{1}_n\ol{\mb{x}}^{t,k}\right\|^2 + 4nL^2\left\|\ol{\mb x}^{t,k}-\mb{x}^*\right\|^2
+ \frac{8L^2\alpha^2}{1-\sigma^2}\mathbb{E}\left[\left\|\mb{y}^{t,k}-\mb{1}_n\ol{\mb{y}}^{t,k}\right\|^2\Big|\mc{F}^{t,k} \right] \nonumber\\
&+ 5L^2\left\|\mb{x}^{t,0}-\mb{1}_n\ol{\mb{x}}^{t,0}\right\|^2
+ 5nL^2\left\|\ol{\mb{x}}^{t,0}-\mb{x}^*\right\|^2
\nonumber\\
&+
4L^2\left\|\mb{x}^{t,k}-\mb{1}_n\ol{\mb{x}}^{t,k}\right\|^2
+ 4nL^2\left\|\ol{\mb{x}}^{t,k}-\mb{x}^*\right\|^2
+ 4L^2\left\|\mb{x}^{t,0}-\mb{1}_n\ol{\mb{x}}^{t,0}\right\|^2
+ 4nL^2\left\|\ol{\mb{x}}^{t,0}-\mb{x}^*\right\|^2
\nonumber\\
=&~
9L^2\left\|\mb{x}^{t,k}-\mb{1}_n\ol{\mb{x}}^{t,k}\right\|^2
+ 8nL^2\left\|\ol{\mb{x}}^{t,k}-\mb{x}^*\right\|^2
+
\frac{8L^2\alpha^2}{1-\sigma^2}\mathbb{E}\left[\left\|\mb{y}^{t,k}-\mb{1}_n\ol{\mb{y}}^{t,k}\right\|^2\Big|\mc{F}^{t,k} \right]
\nonumber\\
&+ 9L^2\left\|\mb{x}^{t,0}-\mb{1}_n\ol{\mb{x}}^{t,0}\right\|^2
+ 9nL^2\left\|\ol{\mb{x}}^{t,0}-\mb{x}^*\right\|^2
\end{align}
We apply the upper bounds on~$U_1,U_2$ in~\eqref{U1} and~\eqref{U2} to derive an upper bound for $\mathbb{E}\left[\left\|\mb{v}^{t,k+1}-\mb{v}^{t,k}\right\|^2\Big|\mc{F}^{t,k}\right]$.
\begin{align}\label{track_error_1}
&\mathbb{E}\left[\left\|\mb{v}^{t,k+1}-\mb{v}^{t,k}\right\|^2\Big|\mc{F}^{t,k}\right]
\leq 2U_1 + 2U_2 \nonumber\\
\leq&~
54L^2\left\|\mb{x}^{t,k}-\mb{1}_n\ol{\mb{x}}^{t,k}\right\|^2 + 38nL^2\left\|\ol{\mb x}^{t,k}-\mb{x}^*\right\|^2
+ \frac{24L^2\alpha^2}{1-\sigma^2}\mathbb{E}\left[\left\|\mb{y}^{t,k} - \mb{1}_n\ol{\mb{y}}^{t,k}\right\|^2\Big|\mc{F}^k\right] 
\nonumber\\
&+ 34L^2\left\|\mb{x}^{t,0}-\mb{1}_n\ol{\mb{x}}^{t,0}\right\|^2
+ 34nL^2\left\|\ol{\mb{x}}^{t,0}-\mb{x}^*\right\|^2
\end{align}
Next, we bound~$2\mathbb{E}\left[\Big\langle W\mb{y}^{t,k}-
W_{\infty}\mb{y}^{t,k}, 
\mb{v}^{t,k+1}-\mb{v}^{t,k}\Big\rangle\Big |\mc{F}^{t,k}\right]$.
\begin{align}
&2\mathbb{E}\left[\Big\langle W\mb{y}^{t,k}-
W_{\infty}\mb{y}^{t,k}, \mb{v}^{t,k+1}-\mb{v}^{t,k}\Big\rangle\Big |\mc{F}^{t,k}\right]\nonumber\\
=&~2\mathbb{E}\left[\Big\langle W\mb{y}^{t,k}-
W_{\infty}\mb{y}^{t,k}, \nabla\mb{f}(\mb{x}^{t,k+1})-\nabla\mb{f}(\mb{x}^{t,k})\Big\rangle\Big |\mc{F}^{t,k}\right]
+ 2\mathbb{E}\left[\Big\langle W\mb{y}^{t,k}-
W_{\infty}\mb{y}^{t,k}, \nabla\mb{f}(\mb{x}^{t,k})-\mb{v}^{t,k}\Big\rangle\Big |\mc{F}^{t,k}\right]\nonumber\\
\triangleq&~Y_1 + Y_2.
\end{align}	
We bound~$Y_1$ and~$Y_2$ separately, starting with~$Y_1$.
\begin{align*}
&Y_1 = \mathbb{E}\left[\Big\langle W\mb{y}^{t,k}-
W_{\infty}\mb{y}^{t,k}, \nabla\mb{f}(\mb{x}^{t,k+1})-\nabla\mb{f}(\mb{x}^{t,k})\Big\rangle\Big |\mc{F}^{t,k}\right] \nonumber\\
\leq&~\frac{1-\sigma^2}{2}\mathbb{E}\left[\left\|\mb{y}^{t,k}-
W_{\infty}\mb{y}^{t,k}\right\|^2\Big|\mc{F}^{t,k}\right]
+ \frac{2\sigma^2}{1-\sigma^2}\mathbb{E}\left[\left\|\nabla\mb{f}(\mb{x}^{t,k+1})-\nabla\mb{f}(\mb{x}^{t,k})\right\|^2\Big | \mc{F}^{t,k}\right].
\end{align*}
We then apply the bound on~$U_2$ in~\eqref{U2} to obtain an upper bound on~$Y_1$. 
\begin{align}\label{Y1}
Y_1\leq&~\frac{1-\sigma^2}{2}\mathbb{E}\left[\left\|\mb{y}^{t,k}-
W_{\infty}\mb{y}^{t,k}\right\|^2\Big|\mc{F}^{t,k}\right]\nonumber\\ 
&+ \frac{2\sigma^2}{1-\sigma^2}\left(18L^2\left\|\mb{x}^{t,k}-\mb{1}_n\ol{\mb{x}}^{t,k}\right\|^2
+ 10nL^2\left\|\ol{\mb{x}}^{t,k}-\mb{x}^*\right\|^2
+ 8L^2\alpha^2\mathbb{E}\left[\left\|\mb{y}^{t,k} - \mb{1}_n\ol{\mb{y}}^{t,k}\right\|^2\Big |\mc{F}^{t,k}\right] \right)\nonumber\\
&+ \frac{16L^2}{1-\sigma^2}\left\|\mb{x}^{t,0}-\mb{1}_n\ol{\mb{x}}^{t,0}\right\|^2
+ \frac{16L^2}{1-\sigma^2}\left\|\ol{\mb{x}}^{t,0}-\mb{x}^*\right\|^2\nonumber\\
\leq&~
\frac{36L^2}{1-\sigma^2}\left\|\mb{x}^{t,k}-\mb{1}_n\ol{\mb{x}}^{t,k}\right\|^2
+ \frac{20nL^2}{1-\sigma^2}\left\|\ol{\mb{x}}^{t,k}-\mb{x}^*\right\|^2
+ \left(\frac{1-\sigma^2}{2}+\frac{16L^2\alpha^2}{1-\sigma^2}\right)\mathbb{E}\left[\left\|\mb{y}^{t,k} - \mb{1}_n\ol{\mb{y}}^{t,k}\right\|^2\Big |\mc{F}^{t,k}\right] \nonumber\\
&+
\frac{16L^2}{1-\sigma^2}\left\|\mb{x}^{t,0}-\mb{1}_n\ol{\mb{x}}^{t,0}\right\|^2
+ \frac{16L^2}{1-\sigma^2}\left\|\ol{\mb{x}}^{t,0}-\mb{x}^*\right\|^2.
\end{align}
Towards~$Y_2$, we first note that:
\begin{align}\label{Y2_0}
Y_2 = 2\mathbb{E}\left[\Big\langle W\mb{y}^{t,k}, \nabla\mb{f}(\mb{x}^{t,k})-\mb{v}^{t,k}\Big\rangle\Big |\mc{F}^{t,k}\right]
-
2\mathbb{E}\left[\Big\langle
W_{\infty}\mb{y}^{t,k}, \nabla\mb{f}(\mb{x}^{t,k})-\mb{v}^{t,k}\Big\rangle\Big |\mc{F}^{t,k}\right].
\end{align}
For the first term in~\eqref{Y2_0}, we have:
\begin{align*}
&2\mathbb{E}\left[\Big\langle W\mb{y}^{t,k}, \nabla\mb{f}(\mb{x}^{t,k})-\mb{v}^{t,k}\Big\rangle\Big |\mc{F}^{t,k}\right] \nonumber\\
=&~2\mathbb{E}\left[\Big\langle W^2\mb{y}^{t,k-1}+W\left(\mb{v}^{t,k}-\mb{v}^{t,k-1}\right), \nabla\mb{f}(\mb{x}^{t,k})-\mb{v}^{t,k}\Big\rangle\Big |\mc{F}^{t,k}\right] \nonumber\\
=&~2\mathbb{E}\left[\Big\langle W\mb{v}^{t,k}, \nabla\mb{f}(\mb{x}^{t,k})-\mb{v}^{t,k}\Big\rangle\Big |\mc{F}^{t,k}\right] \nonumber\\
=&~
2\sum_{i=1}^{n}\mathbb{E}\left[\Big\langle\sum_{r=1}^{n}w_{ir}\mb{v}_r^{t,k},\nabla f_i(\mb{x}_i^{t,k})-\mb{v}_i^{t,k}\Big\rangle\Big|\mc{F}^{t,k}\right] \nonumber\\
=&~
2\sum_{i=1}^{n}w_{ii}\mathbb{E}\left[\Big\langle \mb{v}_i^{t,k},\nabla f_i(\mb{x}_i^{t,k})-\mb{v}_i^{t,k}\Big\rangle\Big|\mc{F}^{t,k}\right]
=
2\sum_{i=1}^{n}w_{ii}\mathbb{E}\left[\Big\langle \mb{v}_i^{t,k}-\nabla f_i(\mb{x}_i^{t,k}),\nabla f_i(\mb{x}_i^{t,k})-\mb{v}_i^{t,k}\Big\rangle\Big|\mc{F}^{t,k}\right]\leq 0
\end{align*}
For the second term in~\eqref{Y2_0}, using the same argument as the above, we have that:
\begin{align*}
&-
2\mathbb{E}\left[\Big\langle
W_{\infty}\mb{y}^{t,k}, \nabla\mb{f}(\mb{x}^{t,k})-\mb{v}^{t,k}\Big\rangle\Big |\mc{F}^{t,k}\right]\nonumber\\
=&~-\frac{2}{n}\sum_{i=1}^{n}\mathbb{E}\left[\Big\langle \mb{v}_i^{t,k}-\nabla f_i(\mb{x}_i^{t,k}),\nabla f_i(\mb{x}_i^{t,k})-\mb{v}_i^{t,k}\Big\rangle\Big|\mc{F}^{t,k}\right]
=\frac{2}{n}\mathbb{E}\left[\left\|\mb{v}^{t,k}-\nabla \mb{f}(\mb{x}^{t,k})\right\|^2\Big|\mc{F}^{t,k}\right].
\end{align*}
Using Lemma~\ref{var_svrg}, we obtain an upper bound on~$Y_2$ as follows:
\begin{align}\label{Y2}
Y_2 \leq
8L^2\left\|\mb{x}^{t,k}-\mb{1}_n\ol{\mb{x}}^{t,k}\right\|^2
+ 8nL^2\left\|\ol{\mb{x}}^{t,k}-\mb{x}^*\right\|^2
+ 8L^2\left\|\mb{x}^{t,0}-\mb{1}_n\ol{\mb{x}}^{t,0}\right\|^2
+ 8nL^2\left\|\ol{\mb{x}}^{t,0}-\mb{x}^*\right\|^2.
\end{align}
Combining the upper bounds on~$Y_1$ and~$Y_2$ in~\eqref{Y1} and~\eqref{Y2}, we obtain:
\begin{align}\label{track_error_2}
&2\mathbb{E}\left[\Big\langle W\mb{y}^{t,k}-
W_{\infty}\mb{y}^{t,k}, 
\mb{v}^{t,k+1}-\mb{v}^{t,k}\Big\rangle\Big |\mc{F}^{t,k}\right] \nonumber\\
\leq&~ Y_1 + Y_2\nonumber\\
\leq&~
\frac{44L^2}{1-\sigma^2}\left\|\mb{x}^{t,k}-\mb{1}_n\ol{\mb{x}}^{t,k}\right\|^2
+ \frac{28nL^2}{1-\sigma^2}\left\|\ol{\mb{x}}^{t,k}-\mb{x}^*\right\|^2
+ \left(\frac{1-\sigma^2}{2}+\frac{16L^2\alpha^2}{1-\sigma^2}\right)\mathbb{E}\left[\left\|\mb{y}^{t,k} - \mb{1}_n\ol{\mb{y}}^{t,k}\right\|^2\Big |\mc{F}^{t,k}\right] \nonumber\\
&+
\frac{24L^2}{1-\sigma^2}\left\|\mb{x}^{t,0}-\mb{1}_n\ol{\mb{x}}^{t,0}\right\|^2
+ \frac{24nL^2}{1-\sigma^2}\left\|\ol{\mb{x}}^{t,0}-\mb{x}^*\right\|^2.
\end{align}
Finally, we obtain the upper bound on~$\mathbb{E}\left[\left\|\mb{y}^{t,k+1} - \mb{1}_n\ol{\mb{y}}^{t,k+1}\right\|^2\Big|\mc{F}^{t,k}\right]$ by combing~\eqref{track_error_1} and~\eqref{track_error_2}.
\begin{align*}
&\mathbb{E}\left[\left\|\mb{y}^{t,k+1} - \mb{1}_n\ol{\mb{y}}^{t,k+1}\right\|^2\Big|\mc{F}^{t,k}\right] \nonumber\\
\leq&~
\frac{98L^2}{1-\sigma^2}\left\|\mb{x}^{t,k}-\mb{1}_n\ol{\mb{x}}^{t,k}\right\|^2 + \frac{66nL^2}{1-\sigma^2}\left\|\ol{\mb x}^{t,k}-\mb{x}^*\right\|^2
+ \left(\frac{1+\sigma^2}{2}+\frac{40L^2\alpha^2}{1-\sigma^2}\right)\mathbb{E}\left[\left\|\mb{y}^{t,k} - \mb{1}_n\ol{\mb{y}}^{t,k}\right\|^2\Big|\mc{F}^k\right] 
\nonumber\\
&+ \frac{58L^2}{1-\sigma^2}\left\|\mb{x}^{t,0}-\mb{1}_n\ol{\mb{x}}^{t,0}\right\|^2
+ \frac{58nL^2}{1-\sigma^2}\left\|\ol{\mb{x}}^{t,0}-\mb{x}^*\right\|^2.
\end{align*}
Taking the total expectation of the above completes the proof.
\end{proof}
\subsection{Main Results}
From Lemma~\ref{consensus_er_svrg}, Lemma~\ref{opt_er_svrg} and Lemma~\ref{track_er_svrg}, we have that: 
if~$0<\alpha\leq\frac{1}{8QL}$, then the following (entry-wise) matrix inequality holds:~$\forall t\geq0$ and~$0\leq k\leq K-1$,
\begin{align*}
\underbrace{\left[
	\begin{array}{l}
	\mathbb{E}\left[\left\|\mb{x}^{t,k+1}-\mb{1}_n\ol{\mb{x}}^{k+1}\right\|^2\right] \\
	\mathbb{E}\left[n\left\|\ol{\mb x}^{t,k+1}-\mb{x}^*\right\|^2\right] \\
	\mathbb{E}\left[\left\|\mb{y}^{t,k+1} - \mb{1}\ol{\mb{y}}^{t,k+1}\right\|^2\right]
	\end{array}
	\right]}_{\triangleq\mb{u}^{t,k+1}}
	\leq& \underbrace{\left[
	\begin{array}{cccc}
	\frac{1+\sigma^2}{2} & 0 & \frac{2\alpha^2}{1-\sigma^2} \\
	\frac{2L^2\alpha}{\mu}&  1-\frac{\mu\alpha}{2} & 0 \\
	\frac{100L^2}{1-\sigma^2} & \frac{70L^2}{1-\sigma^2} & 
	\frac{1+\sigma^2}{2} + \frac{40L^2\alpha^2}{1-\sigma^2}
	\end{array}
	\right]}_{\triangleq G_\alpha}
	\underbrace{\left[
	\begin{array}{l}
	\mathbb{E}\left[\left\|\mb{x}^{t,k}-\mb{1}_n\ol{\mb{x}}^{t,k}\right\|^2\right] \\
	\mathbb{E}\left[n\left\|\ol{\mb x}^{t,k}-\mb{x}^*\right\|^2\right] \\
	\mathbb{E}\left[\left\|\mb{y}^{t,k} - \mb{1}\ol{\mb{y}}^{t,k}\right\|^2\right]
	\end{array}
	\right]}_{\mb{u}^{t,k}} \nonumber\\
	&+\underbrace{\left[
	\begin{array}{cccc}
	0 & 0 & 0 \\
	4L^2\alpha^2&  4L^2\alpha^2 & 0 \\
	\frac{60L^2}{1-\sigma^2} & \frac{60L^2}{1-\sigma^2} & 0
	\end{array}
	\right]}_{\triangleq H_\alpha}
	\underbrace{\left[
	\begin{array}{l}
	\mathbb{E}\left[\left\|\mb{x}^{t,0}-\mb{1}_n\ol{\mb{x}}^{t,0}\right\|^2\right] \\
	\mathbb{E}\left[n\left\|\ol{\mb x}^{t,0}-\mb{x}^*\right\|^2\right] \\
	\mathbb{E}\left[\left\|\mb{y}^{t,0} - \mb{1}\ol{\mb{y}}^{t,0}\right\|^2\right]
	\end{array}
	\right]}_{\mb{u}^{t,0}}.
\end{align*}	
Now we consider the convergence of the above matrix-vector recursion, i.e.,~$\forall t\geq0$ and for~$0\leq k\leq K-1$,
\begin{align}
\mb{u}^{t,k+1} \leq G_\alpha\mb{u}^{t,k} + H_{\alpha}\mb{u}^{t,0}. 
\end{align}	
Applying the above inequality recursively over~$k$ from~$0$ to~$K-1$, we have the following:
\begin{align}\label{recursion_0}
\mb{u}^{t+1,0} 
=\mb{u}^{t,K}
\leq G_{\alpha}\mb{u}^{t,K-1} + H_{\alpha}\mb{u}^{t,0}
\leq \left(G_\alpha^K+\sum_{r=1}^{K-1}G_{\alpha}^{r}H_\alpha\right)\mb{u}^{t,0}.
\end{align} 
where we used the non-negativity of the matrices~$G_\alpha$ and~$H_\alpha$. The following theorem is then straightforward:
\begin{theorem}
	If the step-size~$\alpha$ and the length of inner loop~$K$ are chosen such that~$
	\rho\left(G_\alpha^K+\sum_{r=1}^{K-1}G_{\alpha}^{r}H\right)<1$,
	then~\textbf{\texttt{GT-SVRG}} achieves linear convergence in the outer loop.
\end{theorem}
 Now, we derive the complexity of~\textbf{\texttt{GT-SVRG}} in terms of the total number of local component gradient computations to reach~$\epsilon$-optimal solution, under a specific choice of~$\alpha$ and~$K$.
To do this, we first find the range of the step-size~$\alpha$ such that~$\rho(G_\alpha)<1$, with the help of the following lemma from~\cite{matrix_analysis}.
\begin{lem}\label{rho_bound}
	Let~$A\in\mathbb{R}^{d\times d}$ be a non-negative matrix and~$\mb{x}\in\mathbb{R}^d$ be a positive vector. If~$A\mb{x}\leq\beta\mb{x}$ for~$\beta>0$, then~$\rho(A)\leq\beta$. If~$A\mb{x}<\gamma\mb{x}$ for~$\gamma>0$, then~$\rho(A)<\gamma$.  
\end{lem}
\begin{lem}\label{rho}
If~$0<\alpha\leq\frac{\left(1-\sigma^2\right)^2}{105QL}$, then~$\rho\left(G_\alpha\right)<1$.
\end{lem}
\begin{proof}
In the light of Lemma~\ref{rho_bound}, we solve for a positive constant~$\gamma\geq1$ and a positive vector~$\bds\epsilon= [\epsilon_1,\epsilon_2,\epsilon_3]^\top$ such that ~$G_\alpha\bds\epsilon \leq \left(1-\frac{1}{\gamma}\right)\bds\epsilon$ holds.
This inequality can be expanded as follows:
\begin{align}
\frac{1+\sigma^2}{2}\epsilon_1 + \frac{2\alpha^2}{1-\sigma^2}\epsilon_3 <&~ \left(1-\frac{1}{\gamma}\right)\epsilon_1.
\label{r1}\\
\frac{2L^2\alpha}{\mu}\epsilon_1 + \left(1-\frac{\mu\alpha}{2}\right)\epsilon_2 
 <&~ \left(1-\frac{1}{\gamma}\right)\epsilon_2.
\label{r2}\\
\frac{100L^2}{1-\sigma^2}\epsilon_1 
+ \frac{70L^2}{1-\sigma^2}\epsilon_2
+ \frac{1+\sigma^2}{2}\epsilon_3
+ \frac{40L^2\alpha^2}{1-\sigma^2}\epsilon_3 <&~ \left(1-\frac{1}{\gamma}\right)\epsilon_3.\label{r3}
\end{align}
The above inequalities~\eqref{r1}-\eqref{r3} can written equivalently as follows:
\begin{align}
\frac{1}{\gamma}\leq&~\frac{1-\sigma^2}{2} - \frac{2\alpha^2}{1-\sigma^2}\frac{\epsilon_3}{\epsilon_1}
\label{r1'}\\
\frac{1}{\gamma}\leq&~
\frac{\mu\alpha}{2} - \frac{2L^2\alpha}{\mu}\frac{\epsilon_1}{\epsilon_2} \label{r2'}\\
\frac{1}{\gamma}\leq&~
\frac{1-\sigma^2}{2}
-\frac{100L^2}{1-\sigma^2}\frac{\epsilon_1}{\epsilon_3}
- \frac{70L^2}{1-\sigma^2}\frac{\epsilon_2}{\epsilon_3}
- \frac{40L^2\alpha^2}{1-\sigma^2}\label{r3'}
\end{align}
It can be observed that if the RHS of~\eqref{r1'}-\eqref{r3'} is positive, then we can always find a sufficiently large~$\gamma>1$ such that~\eqref{r1'}-\eqref{r3'} hold. The RHS of~\eqref{r2'} being positive is equivalent to the following:
\begin{align}\label{e12}
\frac{\epsilon_1}{\epsilon_2} < \frac{1}{4Q^2}.
\end{align}  
According to~\eqref{e12}, we set~$$\epsilon_1=1,\qquad\epsilon_2=8Q^2.$$ Towards the RHS of~\eqref{r3'}, we note that if
\begin{align}\label{e3}
\frac{1-\sigma^2}{2}
-\frac{100L^2}{1-\sigma^2}\frac{1}{\epsilon_3}
- \frac{560L^2}{1-\sigma^2}\frac{Q^2}{\epsilon_3} > 0 
\impliedby
\frac{1-\sigma^2}{2}
- \frac{660L^2}{1-\sigma^2}\frac{Q^2}{\epsilon_3} > 0
\end{align}
we can always set~$\alpha$ to be sufficiently small such that the RHS of~\eqref{r3'} is positive. According to~\eqref{e3}, we set
$$\epsilon_3 = \frac{1350Q^2L^2}{(1-\sigma^2)^2}.$$
Now we use the values of~$\epsilon_1,\epsilon_2,\epsilon_3$ in the RHS of~\eqref{r1'}-\eqref{r3'}. For the RHS of~\eqref{r1'} to be positive,
\begin{align}\label{a1}
\alpha^2 < \frac{\left(1-\sigma^2\right)^4}{5400Q^2L^2}
\iff \alpha < \frac{\left(1-\sigma^2\right)^2}{30\sqrt{6}QL}.
\end{align}
For the RHS of~\eqref{r3'} to be positive,
\begin{align}\label{a3}
 \frac{40L^2\alpha^2}{1-\sigma^2}
 <\frac{1-\sigma^2}{2}
 -\frac{100L^2}{1-\sigma^2}\frac{\epsilon_1}{\epsilon_3}
 - \frac{70L^2}{1-\sigma^2}\frac{\epsilon_2}{\epsilon_3} 
 \impliedby
 \frac{40L^2\alpha^2}{1-\sigma^2}
 < \frac{15(1-\sigma^2)}{1350}
\iff \alpha < \frac{3(1-\sigma^2)}{\sqrt{10800}L}.
\end{align}
Therefore, from~\eqref{a1} and~\eqref{a3}, we have that:
$$
\mbox{If}~0<\alpha\leq\frac{\left(1-\sigma^2\right)^2}{105QL},\qquad\rho\left(G_\alpha\right)<1,
$$
which completes the proof.
\end{proof}

Based on Lemma~\ref{rho}, if~$0<\alpha\leq\frac{(1-\sigma^2)^2}{105QL}$, then~$\rho(G_\alpha)<1$, and therefore the following holds:~$\forall t\geq0$ and~$K\geq1$,
\begin{align}\label{recursion}
\mb{u}^{t+1,0} 
\leq \left(G_\alpha^K+\sum_{r=1}^{K-1}G_{\alpha}^{r}H_\alpha\right)\mb{u}^{t,0}
\leq \left(G_\alpha^K+\sum_{r=1}^{\infty}G_{\alpha}^{r}H_\alpha\right)\mb{u}^{t,0}
=\left(G_\alpha^K+\left(I-G_\alpha\right)^{-1}H_\alpha\right)\mb{u}^{t,0},
\end{align} 
To proceed, for a positive vector~$\mb{w}\in\mathbb{R}^d$, we define a weighted matrix norm~$\mn{\cdot}_\infty^{\mb{w}}$, i.e.,~$\mn{A}_\infty^{\mb{w}} = \max_{i}\frac{\sum_{j=1}^{d}|a_{i,j}|w_j}{w_i}$ for~$A=\{a_{i,j}\}\in\mathbb{R}^{d\times d}$, where~$w_i$ is the~$i$th entry of~$\mb{w}$~\cite{bertsekas_PDC_book}. Next, we bound~$\left(I-G_\alpha\right)^{-1}H_\alpha$ under an appropriate matrix norm with the help of the following two lemmas from~\cite{matrix_analysis} and~\cite{bertsekas_PDC_book}.
\begin{lem}\label{rho_monotone}
Let~$X\in\mathbb{R}^{d\times d}$ and~$Y\in\mathbb{R}^{d\times d}$ be some non-negative matrices. If~$X\leq Y$, then~$\rho(X)\leq\rho(Y)$.
\end{lem}
\begin{lem}\label{rho_bound_n}
Let~$A\in\mathbb{R}^{d\times d}$ be a non-negative matrix with~$\rho\left(A\right)<1$. $\forall\epsilon>0$, there exists a positive vector~$\mb{w}\in\mathbb{R}^d$ such that~$\rho\left(X\right)\leq\mn{X}_\infty^{\mb{w}}\leq\rho\left(X\right) + \epsilon$.
\end{lem}
\begin{lem}
If~$\alpha = \frac{(1-\sigma^2)^2}{200QL}$, there exists a weighted matrix norm~$\mn{\cdot}_\infty^{\mb{q}}$ such that
\begin{align*}
\mn{\left(I-G_\alpha\right)^{-1}H_\alpha}_\infty^{\mb{q}} \leq 0.85.
\end{align*}
\end{lem}
\begin{proof}
We first derive an entry-wise upper bound for~$\left(I-G_\alpha\right)^{-1}= \frac{\operatorname{adj}\left(I-G_\alpha\right)}{\det(I-G_\alpha)}$, where~$\operatorname{adj}\left(I-G_\alpha\right)$ is the adjugate matrix of~$I-G_\alpha$. We note that 
\begin{align}
I-G_\alpha = \left[
\begin{array}{cccc}
\frac{1-\sigma^2}{2} & 0 & -\frac{2\alpha^2}{1-\sigma^2} \\
-\frac{2L^2\alpha}{\mu}&  \frac{\mu\alpha}{2} & 0 \\
-\frac{100L^2}{1-\sigma^2} & -\frac{70L^2}{1-\sigma^2} & 
\frac{1-\sigma^2}{2} - \frac{40L^2\alpha^2}{1-\sigma^2}
\end{array}
\right].
\end{align}
We then calculate its determinant as follows.
\begin{align*}
\det(I-G_\alpha)
= \frac{\mu\alpha\left(1-\sigma^2\right)}{4}  \left(\frac{1-\sigma^2}{2}-\frac{40L^2\alpha^2}{1-\sigma^2}\right)
-   \frac{280L^4\alpha^3}{\mu(1-\sigma^2)^2}
- \frac{100L^2\mu\alpha^3}{\left(1-\sigma^2\right)^2}
\triangleq I_1 - I_2 - I_3.
\end{align*}
If~$0<\alpha\leq\frac{1-\sigma^2}{\sqrt{160}L}$, then~$I_1\geq\frac{\mu\alpha\left(1-\sigma^2\right)^2}{16}$. If~$0<\alpha\leq\frac{\left(1-\sigma^2\right)^2\mu}{\sqrt{17920}L^2}$, then~$I_2\leq\frac{\mu\alpha\left(1-\sigma^2\right)^2}{64}$. If~$0<\alpha\leq\frac{\left(1-\sigma^2\right)^2}{80L}$, then~$I_3\leq\frac{\mu\alpha\left(1-\sigma^2\right)^2}{64}$. Combining the bounds on~$I_1,I_2$ and~$I_3$, we obtain a lower bound on~$\det(I-G_\alpha)$:
\begin{align}\label{det}
\text{If}~0<\alpha\leq\frac{\left(1-\sigma^2\right)^2}{\sqrt{17920}QL},\qquad
\det\left(I-G_\alpha\right)\geq\frac{\mu\alpha\left(1-\sigma^2\right)^2}{32}.
\end{align}
Next, we derive an upper bound for~$\operatorname{adj}\left(I-G_\alpha\right)$. Note that since the first row of~$H_\alpha$ is zero, we do not need to compute the first column of~$\operatorname{adj}\left(I-G_\alpha\right)$. We denote~$\left[\operatorname{adj}\left(I-G_\alpha\right)\right]_{i,j}$ as the~$i,j$th entry of~$\operatorname{adj}\left(I-G_\alpha\right)$.
\begin{align}
\left[\operatorname{adj}\left(I-G_\alpha\right)\right]_{1,2}
&= \frac{140L^2\alpha^2}{\left(1-\sigma^2\right)^2},
\qquad
\left[\operatorname{adj}\left(I-G_\alpha\right)\right]_{1,3}
= \frac{\mu\alpha^3}{1-\sigma^2}, \nonumber\\
\left[\operatorname{adj}\left(I-G_\alpha\right)\right]_{2,2}
&\leq\frac{(1-\sigma^2)^2}{4}, \qquad
\left[\operatorname{adj}\left(I-G_\alpha\right)\right]_{2,3}
=\frac{4L^2\alpha^3}{\mu(1-\sigma^2)},\nonumber\\
\left[\operatorname{adj}\left(I-G_\alpha\right)\right]_{3,2}
&= 35L^2, \qquad\qquad
\left[\operatorname{adj}\left(I-G_\alpha\right)\right]_{3,3}
= \frac{\mu\alpha(1-\sigma^2)}{4}. \nonumber
\end{align}
Using the entry-wise upper bound on~$\operatorname{adj}\left(I-G_\alpha\right)$ and the lower bound on~$\det\left(I-G_\alpha\right)$, we obtain an entry-wise upper bound on~$\left(I-G_\alpha\right)^{-1}$ as follows\footnote{We do not compute the first column of~$J_\alpha$ because it is not useful for bounding the spectral radius of~$J_\alpha H_\alpha$.}:
\begin{align*}
\left(I-G_\alpha\right)^{-1} = \frac{\left(I-G_\alpha\right)^*}{\det\left(I-G_\alpha\right)}\leq
\frac{32}{\mu\alpha\left(1-\sigma^2\right)^2}\left[
\begin{array}{cccc}
* & \frac{140L^2\alpha^2}{\left(1-\sigma^2\right)^2} & \frac{\mu\alpha^3}{1-\sigma^2} \\
* &  \frac{(1-\sigma^2)^2}{4} & \frac{4L^2\alpha^3}{\mu(1-\sigma^2)} \\
* & 35L^2 &  \frac{\mu\alpha(1-\sigma^2)}{4}
\end{array}
\right]
=
\left[
\begin{array}{cccc}
* & \frac{4480L^2\alpha}{\mu\left(1-\sigma^2\right)^4} & \frac{32\alpha^2}{(1-\sigma^2)^3} \\
* &  \frac{8}{\mu\alpha} & \frac{128L^2\alpha^2}{\mu^2(1-\sigma^2)^3} \\
* & \frac{1120L^2}{\mu\alpha\left(1-\sigma^2\right)^2} &  \frac{8}{1-\sigma^2}
\end{array}
\right] \triangleq J_\alpha.
\end{align*}
Since~$\mbox{rank}\left(J_\alpha H_\alpha\right)=1$, we have that:
\begin{align*}
\rho\left(J_\alpha H_\alpha\right)
= \mbox{trace}\left(J_\alpha H_\alpha\right)
= \frac{17920L^4\alpha^3}{\mu\left(1-\sigma^2\right)^4} + \frac{1920\alpha^2L^2}{(1-\sigma^2)^4}
+ \frac{32L^2\alpha}{\mu}
+ \frac{7680L^4\alpha^2}{\mu^2(1-\sigma^2)^4}.
\end{align*}
If~$\alpha =  \frac{\left(1-\sigma^2\right)^2}{200QL}$, based on Lemma~\ref{rho_monotone}, we have:
$$
\rho\left(\left(I-G_\alpha\right)^{-1}H_\alpha\right)\leq\rho\left(J_\alpha H_\alpha\right)
\leq \frac{17920}{40000} + \frac{1920}{40000}
 + \frac{6400}{40000} + \frac{7680}{40000} = 0.848.
$$
Since~$\left(I-G_\alpha\right)^{-1}H_\alpha$ is non-negative, based on Lemma~\ref{rho_bound_n}, there exists some weighted maximum norm~$\mn{\cdot}_\infty^{\mb{q}}$ such that~$\mn{\left(I-G_\alpha\right)^{-1}H_\alpha}_\infty^{\mb{q}} \leq 0.848 + 0.002 = 0.85$, which completes the proof.
\end{proof}

Next, we derive an upper bound for~$\rho\left(G_\alpha\right)$ when~$\alpha =  \frac{\left(1-\sigma^2\right)^2}{200QL}$ with the help of~\eqref{r1'}-\eqref{r3'}. 
\begin{lem}\label{rho_G}
If~$\alpha = \frac{\left(1-\sigma^2\right)^2}{200QL}$,~$\rho\left(G_\alpha\right)\leq 1-\frac{\left(1-\sigma^2\right)^2}{800Q^2}$.
\end{lem}
\begin{proof}
We solve for the smallest~$\gamma$ such that~\eqref{r1'}-\eqref{r3'} hold, given that~$\alpha =  \frac{\left(1-\sigma^2\right)^2}{200QL}$ and~$\bds\epsilon = [1,8Q^2,\frac{1350Q^2L^2}{\left(1-\sigma^2\right)^2}]^\top$. From~\eqref{r1'}, we have that:
\begin{align}\label{ga1}
\frac{1}{\gamma}\leq\frac{1-\sigma^2}{2} - \frac{2\alpha^2}{1-\sigma^2}\frac{\epsilon_3}{\epsilon_1}
= \frac{1-\sigma^2}{2} - \frac{2}{1-\sigma^2}\frac{\left(1-\sigma^2\right)^4}{200^2Q^2L^2}\frac{1350Q^2L^2}{(1-\sigma^2)^2}
= \frac{373(1-\sigma^2)}{400}.
\end{align}
From~\eqref{r2'}, we have that:
\begin{align}\label{ga2}
\frac{1}{\gamma}\leq
\frac{\mu\alpha}{2} - \frac{2L^2\alpha}{\mu}\frac{\epsilon_1}{\epsilon_2}
=\frac{\mu\alpha}{4} = \frac{\left(1-\sigma^2\right)^2}{800Q^2}.
\end{align}
From~\eqref{r3'}, we have that:
\begin{align*}
\frac{1}{\gamma}\leq
\frac{1-\sigma^2}{2}
-\frac{100L^2}{1-\sigma^2}\frac{\epsilon_1}{\epsilon_3}
- \frac{70L^2}{1-\sigma^2}\frac{\epsilon_2}{\epsilon_3}
- \frac{40L^2\alpha^2}{1-\sigma^2}
\impliedby
\frac{1}{\gamma}\leq
\frac{1-\sigma^2}{2}
-\frac{660Q^2L^2}{1-\sigma^2}\frac{1}{\epsilon_3}
- \frac{40L^2\alpha^2}{1-\sigma^2}.
\end{align*}
Therefore, we have:
\begin{align}\label{ga3}
&\frac{1}{\gamma}\leq
\frac{1-\sigma^2}{2}
-\frac{660Q^2L^2}{1-\sigma^2}\frac{(1-\sigma^2)^2}{1350Q^2L^2}
- \frac{40L^2}{1-\sigma^2}\frac{(1-\sigma^2)^2}{200^2Q^2L^2}
= \frac{1-\sigma^2}{90} - \frac{1-\sigma^2}{1000Q^2}. \nonumber\\
\impliedby 
&\frac{1}{\gamma}\leq \frac{91\left(1-\sigma^2\right)}{9000}.
\end{align}
Combining~\eqref{ga1},~\eqref{ga2} and~\eqref{ga3}, we have that: if~$\alpha = \frac{\left(1-\sigma^2\right)^2}{200QL}$,
\begin{align}\label{ga}
\frac{1}{\gamma} \leq\min\left\{\frac{373(1-\sigma^2)}{400},\frac{(1-\sigma^2)^2}{800Q^2},\frac{91\left(1-\sigma^2\right)}{9000}\right\} = \frac{(1-\sigma^2)^2}{800Q^2},
\end{align} 
which completes the proof.
\end{proof}
We are now ready to present our main results, that is, the outer loop of~\textbf{\texttt{GT-SVRG}} achieves~$\epsilon$-accuracy in a constant time, up to a logarithmic factor. We note that the weighted matrix norm of the form~$\mn{\cdot}_\infty^{\mb{w}}$ is induced by the weighted maximum vector norm~$\|\cdot\|_{\infty}^{\mb{w}}$, where~$\|\mb{x}\|_\infty^{\mb{w}} = \max_{i}|x_i|/w_i$ for~$\mb{x}\in\mathbb{R}^d$~\cite{matrix_analysis,bertsekas_PDC_book}.
\begin{theorem}\label{cov_outer}
If~$\alpha = \frac{(1-\sigma^2)^2}{200QL}$ and~$K \geq \frac{801Q^2}{(1-\sigma^2)^2}\log(20c)$, where~$c\geq1$ is some constant, the following holds,~$\forall t\geq0$:
$$
\left\|\mb{u}^{t+1,0}\right\|_\infty^{\mb{q}}
\leq 0.9\left\|\mb{u}^{t,0}\right\|_\infty^{\mb{q}}.
$$
\end{theorem}
\begin{proof}
Recall the recursion in~\eqref{recursion}:~$\forall t\geq0$,
\begin{align*}
\mb{u}^{t+1,0} 
\leq\left(G_\alpha^K+\left(I-G_\alpha\right)^{-1}H_\alpha\right)\mb{u}^{t,0}.
\end{align*} 
Since~$\|\cdot\|_\infty^{\mb{q}}$ is a monotone vector norm, we have the following:
\begin{align}\label{rate0}
\left\|\mb{u}^{t+1,0}\right\|_\infty^{\mb{q}}
\leq\mn{G_\alpha^K+\left(I-G_\alpha\right)^{-1}H_\alpha}_{\infty}^{\mb{q}}\left\|\mb{u}^{t,0}\right\|_\infty^{\mb{q}}
\leq \left(\mn{G_\alpha^K}_{\infty}^{\mb{q}}+0.85\right)\left\|\mb{u}^{t,0}\right\|_\infty^{\mb{q}},
\end{align}
where we used Lemma~\ref{rho_G}. Based on Lemma~\ref{rho_G}, when~$\alpha = \frac{\left(1-\sigma^2\right)}{200QL}$, we have~$\rho\left(G_\alpha\right)\leq1-\frac{\left(1-\sigma^2\right)^2}{800Q^2}$. Then from Lemma~\ref{rho_bound_n}, there exists a weighted matrix norm~$\mn{\cdot}_\infty^{\widetilde{\mb{q}}}$ such that~$\mn{G_\alpha}_\infty^{\widetilde{\mb{q}}}\leq 1-\frac{\left(1-\sigma^2\right)^2}{801Q^2}$. Since all norms are equivalent in finite-dimensional vector spaces~\cite{matrix_analysis}, there exists a positive constant~$c\geq1$, such that~$\mn{A}_\infty^{\widetilde{\mb{q}}}\leq c\mn{A}_\infty^{\mb{q}}$ for all~$A\in\mathbb{R}^{3\times3}$.\footnote{The constant~$c$ is greater or equal to~$1$ because every induced matrix norm is minimal~\cite{matrix_analysis}.} We now proceed with~\eqref{rate0} as follows.
\begin{align}\label{rate1}
\left\|\mb{u}^{t+1,0}\right\|_\infty^{\mb{q}}
\leq \left(c\left(1-\frac{\left(1-\sigma^2\right)^2}{801Q^2}\right)^K+0.85\right)\left\|\mb{u}^{t,0}\right\|_\infty^{\mb{q}}\leq\left(c\exp\left\{-\frac{\left(1-\sigma^2\right)^2}{801Q^2}K\right\}+0.85\right)\left\|\mb{u}^{t,0}\right\|_\infty^{\mb{q}},
\end{align}
where we used the inequality~$1+x\leq\exp\{x\},\forall x$. Therefore, we have:
\begin{align*}
\text{If}~K\geq\frac{801Q^2}{\left(1-\sigma^2\right)^2}\log (20c),
\qquad
\left\|\mb{u}^{t+1,0}\right\|_\infty^{\mb{q}}
\leq 0.9\left\|\mb{u}^{t,0}\right\|_\infty^{\mb{q}},
\end{align*}
which completes the proof.
\end{proof}
\textbf{Proof of Theorem~\ref{main_theorem}:} We note that during each inner loop of~\textbf{\texttt{GT-SVRG}}, each node~$i$ computes~$(m_i + 2K)$ local component gradients. Therefore, the total number of local component gradient computations required to achieve~$\epsilon$-accuracy is:
$$
\mc{O}\left(\left(M + \frac{Q^2}{(1-\sigma)^2}\right)\log\frac{1}{\epsilon}\right),
$$  
where~$Q$ is the condition number of the global objective function~$f$,~$1-\sigma$ is the spectral gap of the weight matrix of the graph and~$M$ is the largest number data points at all nodes.

\newpage		
	\bibliographystyle{IEEEbib}
	\bibliography{sample1.bib}
\end{document}